\documentclass[reqno,twoside,12pt]{amsart}
\usepackage{fullpage,amsmath,amsfonts,calrsfs,amssymb,color}

\newtheorem{Theorem}{Theorem}[section]
\newtheorem{Definition}[Theorem]{Definition}
\newtheorem{Proposition}[Theorem]{Proposition}
\newtheorem{Lemma}[Theorem]{Lemma}
\newtheorem{Corollary}[Theorem]{Corollary}
\newtheorem{Remark}[Theorem]{Remark}

\makeatletter
\@addtoreset{equation}{section}

\makeatother

\def\Q{\mathbb Q}
\def\R{\mathbb R}
\def\N{\mathbb N}

\def\E{\mathbb E}
\def\P{\mathbb P}

\def\O{\mathcal O}
\def\eps{\varepsilon}

\def\ds{\displaystyle}
 \def\oo{\mathaccent23}

\newcommand{\one}{1\!\!\!\;\mathrm{l}}

\title[Maximal $L^2$ regularity]{Maximal $L^2$ regularity for Dirichlet problems in  Hilbert spaces}
\date{}

\author[G. Da Prato]{Giuseppe Da Prato}
\address{Scuola Normale Superiore\\
Piazza dei Cavalieri, 7\\ 
56126 Pisa, Italy}
\email{g.daprato@sns.it}

\author[A. Lunardi]{Alessandra Lunardi}
\address{
Dipartimento di Matematica\\
Universit\`a di Parma\\
Parco Area delle Scienze, 53/A\\
43124 Parma, Italy}
\email{alessandra.lunardi@unipr.it}



\begin{document}

\begin{abstract}
We consider the   Dirichlet problem
$\lambda U -  {\mathcal L}U= F$ in $\mathcal O$, $U=0$ on $\partial \mathcal O$. Here 
$F\in L^2(\mathcal O, \mu)$ where $\mu$ is a nondegenerate centered Gaussian measure in a Hilbert space $X$, 
$\mathcal L$ is  an Ornstein--Uhlenbeck operator, and $\mathcal O$ is an  open   set in $X$ with good boundary. We address the problem whether the weak solution 
 $U$ belongs to the Sobolev space $W^{2,2}(\mathcal O, \mu)$. It is well known that the question has positive answer if  $\O = X$; if $\O \neq X$ we give a sufficient condition in terms of geometric properties of the boundary $\partial \O$. The results are quite different with respect to the finite dimensional case, for instance if $\O$ is the ball centered at the origin with radius $r$ we prove that    $U\in W^{2,2}(\mathcal O, \mu)$ only for small $r$.
 \end{abstract}

\maketitle 

\section{Introduction}

 The extension of the rich classical theory of linear elliptic PDE's in finite dimensions to infinite dimensional Hilbert spaces, motivated by  stochastic differential equations  arising in different domains (quantum fields theory, statistical mechanics, biology, chemistry, mathematical finance),  is a widely open field. 

Several well established finite dimensional techniques fail, because of obvious  difficulties such as lack of compactness of bounded closed sets, and of less obvious difficulties such as lack of translation invariant and doubling Borel measures, that prevent  to study equations in Lebesgue and Sobolev spaces using approximation by convolution with mollifiers, singular integrals, and localization methods based on comparison of integrals of functions over balls with integrals over larger balls.  

Some classes of linear elliptic equations and linear parabolic Cauchy problems  have already been studied, mostly with data  in Lebesgue spaces for Gaussian measures or weighted Gaussian measures. In particular, concerning maximal  regularity, the celebrated infinite dimensional Meyer inequalities of \cite{Me}  allow to establish maximal $L^p$ regularity results for Ornstein--Uhlenbeck equations in the whole space, for $1<p<\infty$. See e.g. \cite{Bo}, to which we refer for the general theory of Gaussian measures and Sobolev spaces related to Malliavin Calculus. Maximal $L^2$ (and in some cases $L^p$,   $1<p<\infty$) regularity for different classes of Ornstein--Uhlenbeck equations, with different Sobolev spaces, were proved in \cite{S,DPZ3,CG,MVN}. 
Only a few papers are devoted to differential equations in open sets with boundary conditions. Let us mention \cite{DPGZ,Tal} for Dirichlet problems in spaces of continuous functions and \cite{BDPT1,BDPT2} for Neumann type problems in $L^2$ spaces, all of them for suitable classes of Ornstein--Uhlenbeck operators. 
 
In this paper we begin the study of  maximal $L^2$ regularity result for the Dirichlet problem with an elliptic operator ${\mathcal L}$, 
\begin{equation}
\label{e1e}
\left\{\begin{array}{l}
\lambda U-{\mathcal L} U=F,\quad\mbox{\rm in}\;  \mathcal O,
\\
\\
U=0,\quad\mbox{\rm on}\;\partial\mathcal O,
\end{array}\right.
\end{equation}
where $\lambda >0$ and $F$ are given, and $\O = \{ x\in X:\;G(x)<0\}$ is 
an open set with good boundary. 

Precisely, we fix a centered nondegenerate Gaussian measure $\mu$ in a Hilbert space $X$, we denote by $Q$ its covariance,  and  we consider the Ornstein--Uhlenbeck operator defined on good functions (for instance, smooth cylindrical functions) by
$$\mathcal L U (x) =\frac12\;\mbox{\rm Tr}\;[Q D^2 U(x)]-\frac12\;\langle  x, DU(x)  \rangle. $$
Similarly to the case $\O=X$, a weak solution to \eqref{e1e} is 
a function  $U \in  \oo{W}^{1,2}(\O, \mu)$ such that 
\begin{equation}
\label{sol_debole}
\int_{\O}\lambda U \, \psi \, d\mu + \frac{1}{2} \int_{\O} \langle D_H \varphi,  D_H \psi\rangle _H\,d\mu, 
= \int_{\O}F\,\psi\, d\mu , \quad \forall \psi \in  \oo{W}^{1,2}(\O, \mu).
\end{equation}
Here, $D_H$ is the gradient along the Cameron--Martin space $H= Q^{1/2}(X)$, and  $\oo{W}^{1,2}(\O, \mu)$ is a Sobolev space of functions ``vanishing at the boundary".  

Even in the case of simple open sets such as the unit ball, an explicit basis of $L^2(\O, \mu)$ made by eigenfunctions of $L^{\O}$, that could play the role of the Hermite polynomials and Wiener chaos decomposition used in the case $\O=X$, is not available. Then, we follow a completely different approach, that consists of two steps:
 
\vspace{3mm} {\em Step 1.} We find dimension free  $W^{2,2}$ estimates for finite dimensional problems  approximating   \eqref{e1e}; 

\vspace{3mm} {\em Step 2.}  We approach the weak solution to  \eqref{e1e} by the sequence of cylindrical functions that solve the finite dimensional problems.  

\vspace{3mm} 
Both steps are rather delicate. Let us go into details. 

\vspace{3mm} 
Concerning Step 1, we fix an orthonormal basis $\{e_k:\;k\in \N\}$ of $X$ consisting of eigenvectors of $Q$, $Qe_k = \lambda_ke_k$, and we consider the problems
\begin{equation}
\label{finitedim}
\left\{\begin{array}{l}
\lambda u_n - {\mathcal L}_n u_n = f_n, \quad {\rm in}\;O_n
\\
u_n =0 \quad {\rm at} \;\partial O_n
\end{array}\right. 
\end{equation}
where $f_n(\xi )= F(\sum_{k=1}^{n} \xi_k e_k)$, 
$g_n (\xi) = G(\sum_{k=1}^{n} \xi_k e_k)$, $O_n = \{\xi \in \R^n: \;g_n(\xi) <0\}$, and ${\mathcal L}_n$ is the finite-dimensional Ornstein-Uhlenbeck operator
\begin{equation}
\label{OUinRn}
({\mathcal L}_n u)(\xi)    =\frac12 \sum_{k=1}^{n} \lambda_k D_{kk}u(\xi)  -\frac12\;\sum_{k=1}^{n}\xi_kD_ku(\xi).
\end{equation}
Denoting by  $\mu_n= N_n(x)dx$   the Gaussian measure in $\R^n$ with mean $0$ and covariance matrix $Q_n = $ $diag(\lambda_1, \ldots, \lambda_n)$, we look for an estimate
\begin{equation}
\label{estfinitedim}
\|u_n\|_{W^{2,2}(O_n, \mu_n)}\leq K \|f_n\|_{L^{2}(O_n, \mu_n)}
\end{equation} 
with constant $K$ independent of $n$. Procedures relying on maximal regularity for elliptic operators in $L^p$ spaces with respect to the Lebesgue measure, such as e.g. in \cite{MPPS}, do not work, because the final constant $K$ depends on $n$ in an uncontrollable way. Instead, we follow a more direct approach, which is a  refinement of the approach of \cite{LMP}. Let us explain in the simple case of the unit ball $O_n =B(0,1)$ and $f\in C^{\infty}_c(B(0,1))$. Dimension free bounds for $\|u\|_{W^{1,2}(O_n, \mu_n)}$ are easily found. To estimate the second order derivatives we
differentiate both members of the differential equation in \eqref{finitedim} with respect to $x_h$, we multiply  by $D_hu_n\lambda_h$, we sum up and we integrate by parts, obtaining
$$\int_{B(0,1)} \mbox{\rm Tr}\;[(Q_nD^2u_n)^2] \, d\mu_n = \frac{1}{2}\int_{\partial B(0,1)} \langle D^2u_n  Q\xi, QDu_n\rangle N_n\,ds + \ldots $$
where the dots stand for other integrals that are  under control.  The boundary integral still contains second order derivatives of $u_n$, however using the identities $u_n = f_n =0$ at  $|\xi | =1$ we can express $ \langle D^2u_n  Q_n\xi , Q_nDu_n\rangle $ in terms of first order derivatives of $u_n$, and precisely 
$$ \langle D^2u_n  Q_n\xi , Q_nDu_n\rangle  = \frac{ \langle Q_n^{1/2}Du_n , Q_n^{1/2}\xi  \rangle ^2}{ |Q_n^{1/2}\xi |^2} \bigg(1- \mbox{\rm Tr}\,[Q_n]  + \frac{|Q_n\xi |^2}{|Q_n^{1/2}\xi |^2} \bigg):=
{\mathcal H}_n(\xi )\langle Q_n^{1/2}Du_n , Q_n^{1/2}\xi  \rangle ^2.$$
On the other hand, by a suitable trace lemma it is possible to estimate the boundary integral $\int_{\partial B(0,1)} \langle Q_n^{1/2}Du_n , Q_n^{1/2}\xi  \rangle ^2N_n\,ds $ as
$$\begin{array}{l}
\ds \int_{\partial B(0,1)} \langle Q_n^{1/2}Du_n , Q_n^{1/2}\xi  \rangle ^2N_n\,ds 
\\
\\
\ds \leq C_1 \int_{B(0,1)} \mbox{\rm Tr}\;[(Q_nD^2u_n)^2] \, d\mu_n \|f_n\|_{L^2(B(0,1), \mu_n)} + C_2  \|f_n\|_{L^2(B(0,1), \mu_n)}^2, \end{array}$$
with $C_1$, $C_2>0$ independent of $n$. Therefore, if 
\begin{equation}
\label{hcritico}
\sup_{\xi \in \partial B(0,1)} {\mathcal H}_n(\xi) \leq C_3, \quad \mbox{\rm with }C_3\;\mbox{\rm independent of }n, \end{equation}
we are done: using the Young inequality we get a dimension free estimate for \linebreak $\int_{B(0,1)} \mbox{\rm Tr}\;[(Q_nD^2u_n)^2] \, d\mu_n $ and \eqref{estfinitedim} follows. 

Coming back to general open sets, this procedure leads to the functions 
$${\mathcal H}_n(\xi):= -2\frac{{\mathcal L}_n g_n(\xi)}{ |Q_n^{1/2}g_n(\xi) |^2} + \frac{ D^2g_n(\xi) (Q_nDg_n(\xi),  Q_nDg_n(\xi)) }{ |Q_n^{1/2}Dg_n(\xi) |^4}, \quad \xi \in \partial  O_n .$$
Assuming  that  ${\mathcal H}_n$ is bounded from above on $\partial  O_n $ by a constant independent of $n$, we prove that  \eqref{estfinitedim} holds with $K$ independent of $n$. 
 
Let us discuss the geometrical meaning of this assumption. If we consider the Cameron--Martin scalar product in $\R^n$, $( \xi, \eta):= \langle Q^{-1/2}_n\xi, Q^{-1/2}_n\eta\rangle $, at any point $\xi\in \partial  O_n$ the exterior unit normal vector is just $\nu_n = Q_n Dg_n(\xi)/|Q_n^{1/2}Dg_n|$. The function ${\mathcal H}_n$ turns out to be minus the Gaussian divergence of $\nu_n$, divided by 
$|Q_n^{1/2}Dg_n|$. So, our assumption may be seen as a condition on the ``Gaussian mean curvatures" of the approximating cylindrical sets, hence it is a condition on the ``Gaussian mean curvature" on $\O$. 
Note that this condition is one--sided. Both in finite and in infinite dimensions we do not need that the boundary is uniformly $C^2$. 

Checking this assumption in meaningful examples gives some surprise, and shows important differences between the finite  and the infinite dimensional case. For instance, if $\O $ is the  open ball centered at the origin with radius $r$,  the set $ O_n$ is just  $B(0,r)\subset \R^n$,  and  the suprema $h_n$ of ${\mathcal H}_n$ on the spherical surfaces are bounded by a constant independent of $n$ only if some relationship between $r$ and the eigenvalues of $Q$ is satisfied. In particular if $r^2>$ Tr$\,Q$,  then $\lim_{n\to \infty} h_n = +\infty$ and our condition is not satisfied. See Section \ref{Examples}.

\vspace{2mm}

Let us consider  the second step. It consists in approximating the weak solution $U$ to \eqref{e1e} by the cylindrical functions $U_n$ defined by $U_n(x) = u_n (x_1, \ldots, x_n) $   if $(x_1, \ldots, x_n)   \in O_n$,
$U_n(x) =0$ if $(x_1, \ldots, x_n)\notin O_n$. As usual, here we set  $x_k:=\langle x, e_k\rangle$ for $k\in \N$ and we denote by $P_n$ the orthogonal projection on the subspace   spanned by $e_1 , \ldots  e_n$. The restrictions of $U_n$ to the cylindrical sets  
 $\O_n := \{x\in X:\; G (x_1, \ldots, x_n)<0\} = (G\circ P_n)^{-1}(-\infty, 0)$ are the weak solutions to 
$$\left\{\begin{array}{l}
\lambda U_n-\mathcal L U_n= F\circ P_n,\quad\mbox{\rm in}\;  \mathcal O_n,\\
\\
U_n=0,\quad\mbox{\rm on}\;\partial\mathcal O_n,
\end{array}\right.$$
and  their $W^{1,2}(X, \mu)$ and  $W^{2,2}(\O_n, \mu)$ norms are bounded by a constant independent of $n$, by the finite dimensional estimates. On one hand, this allows to find a subsequence that converges weakly  in $W^{1,2}(X, \mu)$ to a limit function $V$, whose restriction to $\O$ belongs to $W^{2,2}(\O, \mu)$. On the other hand, showing that the restriction of $V$ to $\O$ is the weak solution to \eqref{e1e} is not obvious. If we try to pass to the limit using the definition of weak solution \eqref{sol_debole}, we meet difficulties caused by  the test functions: if $\Phi\in \oo{W}^{1,2}(\O, \mu)$, its standard $n$-dimensional approximations $\varphi_n$ do not vanish in general at  the boundary of $O_n$, and   integrating by parts $(\lambda u_n - {\mathcal L}_n u_n)\varphi_n$ over $O_n$ we obtain a sequence of boundary integrals that is hard to control. 

We overcome this problem using a probabilistic representation formula for the resolvent $R(\lambda, L^{\O})$, recently proved in \cite{DPL}. Such a formula involves the Ornstein--Uhlenbeck process, solution to the  stochastic differential equation
\begin{equation}
\label{e1.6}
dX (t,x)= -\frac{1}{2} X (t,x)dt+Q^{1/2}dW(t),\quad X(0,x)=x, 
\end{equation}
(where $W(t)$ is a  cylindrical Wiener process in $X$, see the Appendix for details),
and its entrance time in the complement of $\overline{\O}$, 
\begin{equation}
\label{e1.7}
\tau_x :=\inf\{t\ge 0:\;X(t,x)\in \overline{\O}^c\},\quad x\in \O .
\end{equation}
Then, the semigroup $T^{\O}(t)$ generated by $L^{\O}$ in $L^2(\O, \mu)$ satisfies 
\begin{equation}
\label{e1.8}
T^{\O}(t) F(x)= \E[F(X (t,x))\one_{\tau_x \ge t}]
= 
 \int_{\{\tau_x \ge t\}}F(X (t,x))d\P,\quad\forall\;x\in \O,
\end{equation}
for each  $F\in C_b(\overline{\O})$, $t>0$. Accordingly, the weak solution $U$ to  \eqref{e1e}, which coincides with the resolvent $R(\lambda, L^{\O})F$, is given by 
\begin{equation}
\label{1.9w}
U(x) = \int_{0}^{+\infty} e^{-\lambda t}(T^{\O}(t) F)(x)dt
\end{equation} 
Using the representation formula \eqref{1.9w}  and tools from the theory of Gaussian measures in Banach spaces, we may pass to the limit if $F\in C_b(X)$ and we prove that the restriction of $V$ to $\O$ is the weak solution to \eqref{e1e}. 
Since the restrictions to $\O$ of functions $F\in C_b(X)$ are dense in $L^2(\O, \mu)$, this concludes the proof. 

However, also this step is not straightforward. Indeed, we rewrite the representation formula \eqref{1.9w}  as 
$$T^{\O}(t) F(x)=  \int_{\Lambda_t}F(\eta(t )) \mu_{t,x}(d\eta),$$
where $ \mu_{t,x}$ is the law of $X(\cdot , x)$ in $C([0, t];X)$, and 
$\Lambda_t = \{ \eta \in C([0, t];X): \;\sup_{0\leq s\leq t} G(\eta(s))\leq 0\}$. Similarly, for each $n\in \N$ we have
$$T^{\O_n}(t) (F\circ P_n)(x)=  \int_{\Lambda_t}F(\eta(t )) \mu_{t,x}^{(n)}(d\eta),$$
where $ \mu_{t,x}^{(n)}$ is the law of $P_nX(\cdot, x) = X(\cdot, P_nx)$ in $C([0, t];X)$. Then we prove that 
$ \mu_{t,x}^{(n)}$ and $\mu$ are Gaussian, that $ \mu_{t,x}^{(n)} \rightharpoonup \mu_{t,x}$ as $n\to \infty$,  and that $\mu_{t,x}(\partial \Lambda_t )=0$. This enables us to pass to the limit in the above representation formula, obtaining that $T^{\O_n}(t) (F\circ P_n)(x)\to T^{\O}(t) F(x)$ for each $x\in \O$ as $n\to \infty$, and then to conclude.

\vspace{2mm}

The paper ends with some examples. We treat the cases of half--spaces and, more generally, of regions below graphs of regular functions, as well as spheres and ellipsoids. In particular, we   
prove that our sufficient condition for maximal Sobolev regularity is satisfied if $\O$ is any half-space, and   if $\O$ is a   ball $B(x_0,r)$ provided a suitable relation between $x_0$, $r$, and the eigenvalues of $Q$ holds. Other examples with  $X= L^2((0,1),dx)$ are sets of the type $\O =\{x\in L^2((0,1),dx): \;\int_0^1 g(x(\xi))d\xi < r\}$ for suitable nonlinear functions $g:\R\mapsto \R$.


\section{Notation and preliminaries}


\subsection{$H$-regular functions} Let  $X$ be a separable real Banach space,  let $\mu$ be a centered non-degenerate Gaussian measure in $X$, and let $H\subset X$ be the associated Cameron--Martin space.

Together with regular functions from $X$ to another Banach  space $E$, we shall consider also $H$-regular functions. 

For $0<\alpha<1$ we say that $F:X\mapsto E$ is locally $\alpha$-H\"older continuous along $H$ if for each $x_0\in X$ there is $r>0$ such that 
$$\sup_{x\in B(x_0, r), \,h\in H\setminus\{0\}} \frac{\|F(x+h)- F(x)\|_E}{|h|^{\alpha}_H}<\infty.$$
A function $F:X\mapsto E$ is $H$-Fr\'echet differentiable at $x_0$ if there exists $L\in {\mathcal L}(H, E)$ such that 
$$\sup_{h\in H,\, |h|_H =1} \|F(x_0+th) - F(x_0) - tLh\|_E = o(t), \quad {\rm as}\;t\to 0.$$
In particular, if $E=\R$ we have $Lh = \langle y, h\rangle_H$ for some $y\in H$, which is denoted by $D_HF(x_0)$. 

$F:X\mapsto \R$ is twice $H$-Fr\'echet differentiable at $x_0$ if it is $H$-Fr\'echet differentiable in a neighborhood $\mathcal N$ of $x_0$ and $D_HF: {\mathcal N}\mapsto {\mathcal L}(H, \R)$  is $H$-Fr\'echet differentiable at $x_0$. The second order derivative is denoted by $D^2_HF(x_0)$. 

If $X$ is a Hilbert space and $Q$ is the covariance of $\mu$, then $H= Q^{1/2}(X)$. So, if 
$F:X\mapsto \R$ is twice  Fr\'echet differentiable, then it is twice  $H$-Fr\'echet differentiable and 
$D_HF(x) = QDF(x)$, $ D^2_HF(x) (h,k) = \langle QD^2F(x)h, k\rangle $, where $DF(x)$ and $D^2F(x)$ are the Fr\'echet first and second order derivatives of $F$ at $x$. 

We say that $F\in C^{2+\alpha}_{H, loc}(X, \R)$ if $F$ is twice  $H$-Fr\'echet differentiable and $D^2_HF$ is  is locally $\alpha$-H\"older continuous along $H$. 

\subsection{Sobolev spaces}

From now on, $X$ is a separable Hilbert space, with inner product $\langle, \cdot, \cdot \rangle$ and norm $\| \cdot \|$. $\mu$ is a centered non-degenerate Gaussian measure in $X$ with covariance $Q$, and $H = Q^{1/2}(X)$ is the associated Cameron--Martin space, endowed with the scalar product  $\langle h_1, h_2\rangle _H:= \langle Q^{-1/2}h_1,  Q^{-1/2}h_2\rangle$ and the associated norm $|h|_H := \|Q^{-1/2}h\|$.

Let us recall the definition and some properties of the Sobolev spaces that we need. For the general treatment of Gaussian measures we refer to \cite{Bo}. 

We fix once and for all an orthonormal basis of $X$ consisting of eigenvectors of $Q$. We consider an ordering of the basis such that $Q e_k = \lambda_ke_k$, for each $k\in \N$, and  the sequence $(\lambda_k)$ decreases. 
We set
$$x_k := \langle x, e_k\rangle, \quad k\in \N,  \;x\in X, \;k\in \N .$$
If $F: X\mapsto \R$ is a differentiable function, for every $k\in \N$ and $x\in X$ we denote by $D_kF(x)$ its derivative in the direction $e_k$. So, $D_kF(x) =\langle  D F(x), e_k\rangle$.  

The Sobolev spaces $W^{1,2}(X, \mu)$ and $W^{2,2}(X, \mu)$ may be defined in several ways. We 
recall here the definition through the weak derivatives. 

Let $U\in L^2(X, \mu)$. We say that a function $F\in L^1(X, \mu)$ is the  weak derivative of $U$ in the direction of $e_k$ if for every $\Psi\in C^1_b(X)$ (the space of the bounded continuously differentiable functions from $X$ to $\R$ with bounded gradient) we have
$$\int_X U\,D_k\Psi\,d\mu = -\int_X F\,\Psi\, d\mu + \frac{1}{\lambda_k} \int_X x_k U\,\Psi\,d\mu .$$
In this case, we still set $F=D_kU$. The space $W^{1,2}(X, \mu)$ is the set of all $U\in L^2(X, \mu)$ having 
weak derivatives $D_kU \in L^2(X,\mu)$ and such that 
$$\| U \|_{W^{1,2}(X, \mu)} := \bigg( \int_{X} |U|^2 d\mu +  \int_{X} \sum_{k=1}^{\infty} \lambda_k |D_kU|^2 d\mu\bigg)^{1/2} <\infty .$$
Similarly, the space $W^{2,2}(X, \mu)$ is the set of all $U\in W^{1,2}(X, \mu)$ such that every weak derivative 
$D_kU$ is weakly differentiable in every direction  $e_h$, and denoting by $D_{hk}U$ the second order weak derivatives, $u$  satisfies
$$\| U \|_{W^{2,2}(X, \mu)} := \bigg( \| U \|_{W^{1,2}(X, \mu)}^2 +  \int_{X} \sum_{h,k=1}^{\infty} \lambda_k \lambda_h |D_{hk}U|^2 d\mu \bigg)^{1/2} <\infty .$$
$W^{1,2}(X, \mu)$ and $W^{2,2}(X, \mu)$ are Hilbert spaces, with the scalar products
$$\langle U, V\rangle _{W^{1,2}(X, \mu)} :=  \int_{X} U \, V\,  d\mu +  \int_{X} \sum_{k=1}^{\infty} \lambda_k D_kU \,D_kV\,d\mu , $$
$$\langle U, V\rangle _{W^{2,2}(X, \mu)} := \langle U, V\rangle _{W^{1,2}(X, \mu)} +    \int_{X} \sum_{h, k=1}^{\infty} \lambda_k \lambda_h D_{hk}U \,D_{hk}V\,d\mu. $$
Note that $\sum_{k=1}^{\infty} \lambda_k  |D_kU(x)|^2 = |D_H U(x)|_H^2$ if $U$ is differentiable at $x$ along $H$, and \linebreak
$ \sum_{h,k=1}^{\infty} \lambda_k   \lambda_h |D_{hk}U(x)|^2 = \|D^2_HU(x)\|_{HS}^2$ if $U$ is twice differentiable at $x$ along $H$, where we denote by $\|\cdot\|_{HS}$ the Hilbert-Schmidt norm on the bilinear Hilbert-Schmidt mappings from $H$ to $\R$. 

$W^{1,2}(X, \mu)$ and $W^{2,2}(X, \mu)$  coincide with the spaces $D^{2,1}(\mu, \R)$ and $D^{2,2}(\mu, \R)$ of \cite{Bo}, respectively  (\cite[Cor. 5.4.7]{Bo}), hence they coincide with the usual Sobolev spaces of the Malliavin calculus  (\cite[Thm. 5.7.2]{Bo}). 
We refer to \cite{M,Bo} for   systematic treatments of such Sobolev spaces.

If $\O$ is any open subset of $X$, the definition of the weak derivatives and of the spaces  $W^{1,2}(\O, \mu)$ and $W^{2,2}(\O, \mu)$ is similar, the only difference being that  the test functions $\Psi\in C^1_b(\overline{\O})$ should vanish at the boundary $\partial \O$, or equivalently in a neighborhood of $\partial \O$. It is not hard to see that 
$W^{1,2}(\O, \mu)$ and $W^{2,2}(\O, \mu)$ are complete. 

To treat Dirichlet problems, we need functions that vanish at the boundary, in some suitable sense. 
 If  $U:\O \mapsto \R$, we define its null extension to the whole $X$ by $\widetilde{U}$, namely we set 
$$\widetilde{U}(x)=\left\{\begin{array}{ll}
U(x), & \mbox{\rm if}\;x\in \O,\\
0,& \mbox{\rm if}\;x\notin \O.
\end{array}\right.$$

\begin{Definition}
\label{Def:SobNulli}
We denote by $\oo{W}^{1,2}(\O, \mu)$ the set of the functions $U\in L^2(\O, \mu)$ whose null extension $\widetilde{U}$ belongs to  $W^{1,2}(X, \mu)$. It is endowed with the scalar product $\langle U,V\rangle _{\oo{W}^{1,2}(\O, \mu)} := \langle \widetilde{U} ,\widetilde{V}\rangle_{W^{1,2}(X, \mu)}$. 
\end{Definition}

Since  $W^{1,2}(X, \mu)$ is a Hilbert space, then $ \oo{W}^{1,2}(\O, \mu)$ is a Hilbert space, too. 
 

\section{Finite-dimensional estimates}
\label{sect:finite}

In this section we consider problems in open subsets of $\R^n$, with  fixed $n$. $\mu$ is  the Gaussian measure in $\R^n$ with mean $0$ and covariance $Q := $diag$(\lambda_1, \ldots, \lambda_n)$, and ${\mathcal L}$ is the associated Ornstein-Uhlenbeck operator, i.e.
$${\mathcal L}\varphi (x) = \frac{1}{2} \sum_{k=1}^{n}\lambda_k D_{kk}\varphi(x) -  \frac{1}{2} \sum_{k=1}^{n}x_kD_k\varphi(x).$$
 We consider weak solutions to Dirichlet problems in an open set 
$$O= \{x\in \R^n:\;g(x) <0\},$$
where $g$ is a smooth ($C^{2+\alpha}_{loc}(\R^n)$, for some $\alpha >0$) function whose gradient does not vanish at the boundary  $\partial O$. 
Moreover we need other conditions guaranteeing global regularity of $\partial O = g^{-1}(0)$. Precisely, we assume that there exists $ \delta >0$ such that 
\begin{equation}
\label{Ondelta}
  \sup_{| g(x)|\leq \delta}  |Q^{1/2}D  g| :=a <\infty, \;  \sup_{| g(x)|\leq \delta} \| Q^{1/2} D^2 g(x) Q^{1/2}\|_{{\mathcal L}(\R^n)}:=b <\infty .
\end{equation}
(Here and in the following, ${\mathcal L}(\R^n)$ is endowed with the Hilbert-Schmidt norm, 
$\|M\|_{{\mathcal L}(\R^n)} = ( \sum_{i, j=1}^{n} m_{ij}^2)^{1/2}$ if $ m_{ij}= \langle Me_i, e_j\rangle$).

Assumption \eqref{Ondelta} yields $g\in C^2_b( g^{-1}([-\delta, \delta]))$, it is obviously verified if $O$ is bounded, and it implies that   $\partial O$ is the level set of a modification of $g$ that satisfies dimension free bounds in the whole $O$, as the next  lemma shows.

\begin{Lemma}
\label{Le:g}
Let $O= \{x\in \R^n:\; g(x) <0\}$ where $ g$ is a  $C^{2+\alpha}_{loc}$ function satisfying  \eqref{Ondelta}, and whose gradient does not vanish at  $ g^{-1}(0)$.
 Then there exists a $C^{2+\alpha}_{loc}$ function $\widetilde{g}$ such that $O= \{x\in \R^n:\;\widetilde{g}(x) <0\}$, $ g= \widetilde{g}$ in a neighborhood of $\partial O$, and 
\begin{equation}
\label{A,B}
\begin{array}{l}
A:=  \sup_{x\in O} |Q^{1/2}D\widetilde{g}(x)| \leq a,
\\
\\
 B:=  \sup_{x\in O} \| Q^{1/2} D^2\widetilde{g}(x) Q^{1/2}\|_{{\mathcal L}(\R^n)} \leq b+ 3a^2/\delta .
 \end{array}
\end{equation}
\end{Lemma}
\begin{proof} Let $\eta \in C^{\infty}(\R)$ be an odd nondecreasing function such that $\eta(r) = r$ for $0\leq r\leq \delta/2$, $\eta (r)=\delta $ for $r\geq \delta$, and $\|\eta'\|_{\infty} \leq 1$, $\|\eta ''\|_{\infty} \leq 3/\delta $. 
Set  $\widetilde{g}  = \eta \circ g$. Then
$$D_i \widetilde{g} =  (\eta ' \circ g)D_ig, \quad D_{ij}g = ( \eta ' \circ g)D_{ij}g
+ ( \eta '' \circ g)D_i gD_j g, \quad i,j=1, \ldots n, $$
so that  $\widetilde{g}$ enjoys the claimed properties. \end{proof}

Let us introduce a weighted surface measure on $\partial O$:
\begin{equation}
\label{Lebesgue} 
d\sigma  = N (x)\frac{|Q ^{1/2}D g(x)|}{|D g(x)|} ds
\end{equation}
where $N (x) = (2\pi \det Q )^{-1/2}\exp( -|Q ^{-1/2}x|^2/2)$ is the Gaussian weight and $ds$ is the usual surface Lebesgue measure. So, $\sigma$ is a weighted Lebesgue surface measure, independent of $g$ since  if $\partial O$ is a level surface of another good function $g_1$, then   $D g_1$ is a scalar multiple of $D g$ at any point of the boundary. Moreover, we have 
$$\frac{1}{\lambda_1^{ 1/2}} \leq \frac{|D g(x)|}{|Q ^{1/2}D g(x)|} \leq 
\frac{1}{\lambda_n^{ 1/2}}, \quad x\in \partial O ,$$
so that the spaces $L^p(\partial O, d\sigma)$ are equivalent to the spaces $L^p(\partial O, N\,ds)$, but one of the equivalence constants blows up as  $n\to \infty$ and this is important in view of the infinite dimensional case. 

Using the surface measure $d\sigma$ the integration by parts formula reads as 
\begin{equation}
\label{parti} 
\int_{O } D_k\varphi \,\psi \,d\mu = -\int_{O } \varphi \,D_k\psi \,d\mu +  \frac{1}{\lambda_k} \int_{O } x_k\varphi \,\psi  \,d\mu + \int_{\partial O } \frac{D_k g}{|Q^{1/2}D g| }\varphi \,\psi \,d\sigma , 
\end{equation}
for each $\varphi$, $\psi \in  W^{1,2}(O, \mu)$, one of which with bounded support. Indeed, in  this case the boundary integral is meaningful, since $W^{1,2}(O, \mu)$ $\subset $ $W^{1,2}_{loc}(O,dx)$ so that  the trace at the boundary of any function in $W^{1,2}(O, \mu)$ belongs to $L^2_{loc}(\partial O, ds) =  L^2_{loc}(\partial O, d\sigma)$.

Applying \eqref{parti} with $\varphi$ replaced by $\lambda_k  D_k\varphi $ and summing up, we find
\begin{equation}
\label{partiL}
\begin{array}{l}
\ds{ \int_{ O }{\mathcal L}\varphi \, \psi \, d\mu = }
\\
\\
= \ds{
-  \frac{1}{2} \int_{ O } \langle Q^{1/2}D \varphi, Q^{1/2}D\psi \rangle d\mu + \frac{1}{2}\int_{\partial  O } \frac{  \langle Q^{1/2}D\varphi, Q^{1/2}D g \rangle }{|Q^{1/2}D g| }\psi \, d\sigma }, 
\end{array}
\end{equation}
for every $\varphi \in W^{2,2}(O , \mu)$, $\psi\in W^{1,2}(O , \mu)$, one of which  with bounded support.


\subsection{Maximal Sobolev regularity}
\label{Sobolev}

In this section we give dimension free estimates for the 
weak solution $u\in \oo{W}^{1,2}(O, \mu)$  to 
 \begin{equation}
\label{Diri_n}
\left\{\begin{array}{l}
\lambda u-{\mathcal L} u  = f,\quad\;\mbox{\rm in}\; O,\\
\\
u=0,\quad\;\mbox{\rm on}\;\partial  O.
\end{array}\right.
\end{equation}
with $\lambda >0$ and $f\in   L^2(O, \mu)$. Since $C^{\infty}_c(O)$ is dense in $ L^2(O, \mu)$ by \cite[Lemma 2.1]{LMP}, we may assume that 
$f\in C^{\infty}_c(O)$. 
Since the Gaussian measure $\mu$ is locally equivalent to the Lebesgue measure, and the boundary of $O$ is smooth enough,   the standard regularity results about elliptic equations imply that  $u\in C^2(\overline{O} )$, it is smooth in $O$, and it is a classical solution to \eqref{Diri_n}. 
 In particular, 
$u\in  W^{2,2}_{loc}( O, \mu)$, that is for   every ball $B\subset \R^n$ with nonempty intersection with $O$, the restriction of $u$ to $O\cap  B$ belongs to $ W^{2,2} (O\cap B, \mu)$. 
Since $O$ is possibly unbounded, regular functions   do not necessarily belong to   $W^{2,2}( O, \mu)$. So, we   introduce a smooth cutoff function $ \theta :\R^n \to \R$   such that 
$$0\leq \theta (x)\leq1, \;\; |D \theta(x)| \leq 2, 
\;\; \theta \equiv 1 \;\mbox{\rm  in}\; B(0,1), \;\; \theta \equiv 0 \;\mbox{\rm outside}\; B(0,2) $$
and we set, for $R>0$,  
$$\theta_R(x) = \theta(x/R), \quad x\in \R^n.$$

To begin with, $W^{1,2}$ estimates are easy. Taking $u$ as a test function in the definition of weak solution, we get 
$$\lambda \int_O u^2d\mu + \frac{1}{2} \int_O |Q^{1/2}D u|^2 d\mu = \int_O f u\, d\mu, $$
so that
 \begin{equation}
\label{stimaX1}
(i)\; \int_Ou^2 d\mu \leq \frac{1}{\lambda^2}\|f\|_{L^2(O, \mu)}^2, \quad (ii)\;  \int_O |Q^{1/2}D u|^2 d\mu \leq \frac{2}{\lambda} \|f\|_{L^2(O, \mu)}^2.
 \end{equation}

Localized estimates that involve the second order derivatives of $u$ are given by the following proposition.

\begin{Proposition}
\label{p3.1}
For every   $f\in C^{\infty}_{c}(O)$ and $\eps \in (0, 1)$ there exists $R_0$ such that for $R>R_0$ the solution $u$ to \eqref{Diri_n} satisfies 
\begin{equation}
\label{e3.1}
\begin{array}{l}
\ds 
\bigg(  \frac12 -\frac{\eps}{2} \bigg) \int_O \theta_R^2\mbox{\rm Tr}\;[(Q D^2u)^2]d\mu 
 \\
\\
\ds \leq \bigg( 4 + \frac{\eps}{\lambda}\bigg) \|f\|^{2}_{L^2(\O, \mu)} +  \frac12\int_{\partial O} \theta_R^2  \frac{\langle D^2u\cdot QD g,  QD u  \rangle }{|Q^{1/2}D g| }\, d\sigma.
\end{array} 
\end{equation}
\end{Proposition}
\begin{proof} Recall that $u$ is smooth in $O$. Differentiating \eqref{Diri_n} with respect to $x_h$  yields
$$
\lambda D_hu -{\mathcal L}  D_hu +\frac12 D_hu =D_hf.
$$
We would like to multiply  both sides by $\lambda_hD_hu$, sum over $h$ and integrate over $O$, using the integration formula \eqref{partiL}.  However, if $O$ is unbounded we do not know {\em a priori} whether what comes out  belongs to $L^1(O, \mu)$,  and this is why we introduce the cutoff functions  $\theta_R$.  Multiplying by $\lambda_hD_hu\,\theta_R^2$
we obtain  
$$
\bigg( \lambda + \frac{1}{2}\bigg) \lambda_h(D_hu)^2\theta_R^2 -\lambda_h {\mathcal L} D_hu \cdot D_hu\,\theta_R^2  = \lambda_h D_hf\,D_hu\,\theta_R^2.
$$
Integrating over $O$ and using \eqref{partiL} yields
$$
\begin{array}{l}
\ds \int_O \bigg( \lambda + \frac{1}{2}\bigg)\lambda_h |D_hu|^2\theta_R^2 \, d\mu
+\frac12\int_O\lambda _h|Q^{1/2}D D_hu|^2 \theta_R^2 \, d\mu
\\
\\
\ds
+ \int_O \lambda_h \theta_R\langle  Q^{1/2}D (D_hu), Q^{1/2}D \theta_R\rangle D_hu \,d\mu
\\
\\
\ds =
\frac12\int_{\partial O} \frac{\lambda_h \langle Q^{1/2}D D_hu,Q^{1/2}D  g  \rangle  D_hu}{ |Q^{1/2}D g| }\, d\sigma + \int_O \lambda_hD_hf D_hu\,   \theta_R^2 \, d\mu.
\end{array}
$$
 Summing   over $h$ we find 
$$
\begin{array}{l}
\ds \int_O \bigg( \lambda + \frac{1}{2}\bigg) |Q^{1/2}D u|^2 \theta_R^2\, d\mu 
+ \frac12 \int_O\mbox{\rm Tr}\;[(Q D^2u)^2]\theta_R^2  \, d\mu
\\
\\
\ds
+  \int_O \langle Q^{1/2}D^2uQD u, Q^{1/2}D  \theta_R\rangle  \theta_R  \, d\mu
\\
\\
\ds = \frac12 \int_{\partial O} \frac{\langle Q^{1/2}D^2u\,QD u, Q^{1/2}Dg\rangle }{ |Q^{1/2}D g| }\, \theta_R^2 \,d\sigma + \int_O \langle Q^{1/2}D f,Q^{1/2}D u  \rangle  \theta_R^2 \,d\mu.
\end{array} 
$$
Since $f$ has compact support, for $R$ large enough $\theta_R \equiv 1$ on
the support of $f$. For such $R$ we use  again \eqref{partiL} in the last integral, and recalling \eqref{stimaX1}(i)  we obtain
$$\bigg|  \int_O \langle Q^{1/2}D f,Q^{1/2}D u  \rangle  \theta_R^2 \,d\mu\bigg| = 
\bigg| -2\int_{ O} f(\lambda u - f)\,d\mu\bigg| \leq 4 \|f\|^2_{L^2(O, \mu)}.$$
Moreover,
$$\begin{array}{l}
\ds| \int_O \langle Q^{1/2}D^2u \, QD u, Q^{1/2}D  \theta_R\rangle  \theta_R  \,d\mu| 
\\
\\
\ds \leq 
\int_O \mbox{\rm Tr}\;[(Q D^2u)^2] \,|Q^{1/2}D\theta_R|\, |Q^{1/2}D u|\,\theta_R\,d\mu
 \\
 \\
\ds \leq \bigg( \int_O\mbox{\rm Tr}\;[(Q D^2u)^2]\theta_R^2  d\mu\bigg)^{1/2} \bigg( \int_O  |Q^{1/2}D u|^{1/2}\, d\mu\bigg)^{1/2}\frac{\|\,|D \theta|\,\|_{\infty}}{R}
\\
\\
\ds \leq \bigg( \frac{1}{2}  \int_O\mbox{\rm Tr}\;[(Q D^2u)^2]\theta_R^2  \, d\mu 
+  \frac{1}{\lambda}  \|f\|^2_{L^2(O, \mu)}\bigg) \frac{\|\,|D \theta|\,\|_{\infty}}{R}, 
\end{array}$$
where  \eqref{stimaX1}(ii) has been used in the last step. Taking  $R$ large enough, such that $ \|D \theta\|_{\infty} / R  \leq \varepsilon$, the statement follows.
\end{proof}

Note that if $O$ is bounded we do not need the cutoff functions and instead of \eqref{e3.1} we obtain  an equality, 
\begin{equation}
\label{e3.1bounded}
\begin{array}{l}
\ds \bigg( \lambda +\frac12\bigg) \int_O|Q^{1/2}D u|^2 d\mu
+ \int_O\mbox{\rm Tr}\;[(Q D^2u)^2]d\mu 
\\
\\
\ds = 2\int_O (f-\lambda u) f\,d\mu +  \frac12\int_{\partial O} \frac{\langle D^2u \cdot QD g,  QD u  \rangle}{ |Q^{1/2}D g| }  d\sigma.  
\end{array} 
\end{equation}

The next step is devoted to the boundary integral
$$J_R := \int_{\partial O}  \frac{\langle D^2u \cdot QD g,  QD u \rangle}{ |Q^{1/2}D g|}\theta_R^2 \, d\sigma .$$
Using the fact that $u$ and ${\mathcal L}u$ vanish at the boundary, we shall rewrite $J_R$ as an integral that involves only first order derivatives of $u$.

\begin{Lemma}
\label{Le:identita}
Let $u\in C^2(\overline{O})$ be such that both $u$ and ${\mathcal L}u$
vanish at $\partial O$. Then, 
$$\langle D^2u \cdot  QD g,  QD u \rangle =  \langle Q^{1/2}D u,   Q^{1/2}D g\rangle ^2  \,h \quad  \mbox{\rm at}\; \partial O, $$
where
\begin{equation}
\label{hn}
h(x) :=  -2\frac{ {\mathcal L}g(x)}{ |Q^{1/2}D g(x) |^2} + \frac{ \langle  D^2g(x) Q D g(x), Q  D g(x)\rangle }{ |Q^{1/2}D g(x) |^4}. 
\end{equation}
\end{Lemma}
\begin{proof}
There exists a neighborhood ${\mathcal U}$ of $\partial O$ such that for every $x\in  {\mathcal U}$ the
distance $d(x, \cdot) : \partial O \mapsto \R$, $y\mapsto d(x,y)$, has a unique minimum point $y = P(x)$, and the mapping $x\mapsto P(x)$ is differentiable in ${\mathcal U } $. Moreover, 
\begin{equation}
\label{e4c}
P'(x)=I-\frac{D g(x) \otimes  D g(x)}{|D g(P(x))|^{ 2}},\quad x\in \partial O,
\end{equation}
where we have used  the standard notation
 $(v\otimes w)z= \langle w,z  \rangle v$. 

Since $u$ vanishes at the boundary, its gradient is a scalar multiple of $D g $ at each point of $\partial O$. Therefore, for each $x$ in a neighborhood of   $\partial O$ in $O$ we have
$$\langle D u(P(x)), k\rangle = \beta(P(x))\langle D g(P(x)), k\rangle , \quad k\in \R^n, $$
where $\beta = \langle D u, D g\rangle/ |D g|^2$. Differentiating we obtain  
\begin{equation}
\label{e4d}
\begin{array}{l}
\langle D^2 u (P(x))P'(x)k, l\rangle    =    \langle D \beta (P(x)), P'(x)k\rangle \,\langle D g (P(x)), l\rangle    
\\
\\
  + \beta (P(x)) \langle D^2g(P(x))P'(x)k, l\rangle, \quad   l, \;k\in \R^n. \end{array}
  \end{equation}
Let us compute $D \beta$ at $\partial O$. Since 
$$\langle D \beta, l\rangle = \frac{ \langle D^2u \cdot  l, D g \rangle + \langle D u, D^2g\cdot l\rangle }{|D g|^2}
- 2\frac{\langle  D u,  D g\rangle \langle D^2g \cdot D u, l\rangle }{| D g|^4}$$
recalling that $D u = \beta D g$ at the boundary, we obtain 
\begin{equation}
\label{e1ebis}D \beta = \frac{D^2u D g - \beta D^2g D g}{|D g|^2}, \quad {\rm at}\;\partial O.
\end{equation}
Replacing \eqref{e1ebis} and \eqref{e4c} in \eqref{e4d}, for each $x\in \partial O$  we get 
\begin{equation}
\label{e4f}
\begin{array}{l}
\langle (D^2u -\beta D^2g) k, l\rangle  = 
\\
\\
= \ds{\langle (D^2u -\beta D^2g)D g, l\rangle  \frac{\langle D g, k\rangle}{|D g|^2} + \langle (D^2u -\beta D^2g)D g, k\rangle  \frac{\langle D g, l\rangle}{|D g|^2}}
\\
\\
\ds{-  \langle (D^2u -\beta D^2g)D g, D g\rangle  \frac{\langle D g, k\rangle \langle D g, l\rangle}{|D g|^4}, \quad l, \;k\in \R^n. }
\end{array}
\end{equation}
Taking $l=k=QD g$ we get 
\begin{equation}
\label{formaquad}
\begin{array}{l}
\langle (D^2u -\beta D^2g)QD g, QD g\rangle  =
2  \langle (D^2u -\beta D^2g)D g, QD g\rangle \ds{ \frac{|Q^{1/2}D g|^2}{|D g|^2} }
\\
\\
-  \langle (D^2u -\beta D^2g)D g,  D g\rangle \ds{ \frac{|Q^{1/2}D g|^4}{|D g|^4}}.
\end{array}  
\end{equation}
Taking now $l=k= Q^{1/2}e_i = \lambda_i^{1/2}e_i$ and summing over $i$, we get 
\begin{equation}
\label{traccia}
\begin{array}{l}
\mbox{\rm Tr}[Q (D^2u - \beta D^2g )  ]  = 2  \langle (D^2u -\beta D^2g)D g, QD g\rangle \ds{ \frac{1}{|D g|^2}}
\\
\\
 -  \langle (D^2u -\beta D^2g)D g,  D g\rangle \ds{ \frac{|Q^{1/2}D g|^2}{|D g|^4} }. 
 \end{array}
\end{equation}
Comparing \eqref{formaquad} and \eqref{traccia} yields 
$$\langle (D^2u -\beta D^2g)QD g, QD g\rangle  = |Q^{1/2}D g|^2 \mbox{\rm Tr}[Q (D^2u  - \beta D^2g )  ] . $$
Replacing $ \mbox{\rm Tr}[Q D^2u] = \langle x, D u\rangle = \beta \langle x, D g\rangle $ we obtain
$$\langle  D^2u \cdot QD g, QD g\rangle  =  \beta  ( \langle D^2g QD g, QD g\rangle  
-2 |Q^{1/2}D g|^2{\mathcal L}g)$$
and the statement follows multiplying both sides by $\beta = \langle Q^{1/2}D u,  Q^{1/2}D g\rangle/ | Q^{1/2}D g|^2$. 
\end{proof}

To estimate the boundary integral $J_R$ we shall use   the following lemma.

\begin{Lemma}
\label{dernorm}
Let \eqref{Ondelta} hold, and let $A$, $B$ be defined by  \eqref{A,B}. Then for every $R>0$, $\varphi  \in W^{2,2}_{loc}( O, \mu)\cap W^{1,2}( O ,\mu)$  we have
\begin{equation}
\label{DerNormR}
\begin{array}{l}
\ds{ \int_{\partial  O} \frac{ \langle Q^{1/2}D \varphi, Q^{1/2}D g \rangle ^2}{|Q^{1/2}D g|}\, \theta_R^2 \, d\sigma }
\\
\\
\ds{ \leq   \bigg(2 \|\theta_R^2{\mathcal L}\varphi \|_{L^2(O,\mu)} + \big( \int_{ O} \theta_R^2 \mbox{\rm Tr}\;[(Q D^2\varphi)^2]d\mu \big)^{1/2}
\bigg)   \|\,|Q^{1/2}D \varphi |\,\|_{L^2(O,\mu)} A }
\\
\\
\ds{ + \bigg(  \frac{4\lambda_1^{1/2} }{R}\,A +B \bigg)    \|\,|Q^{1/2}D \varphi|\,\|_{L^2(O,\mu)} ^2.}
\end{array}
\end{equation}
\end{Lemma}
\begin{proof} Take $\psi  =  \langle Q^{1/2}D \varphi, Q^{1/2}D \widetilde{g} \rangle \theta_R^2$
in  formula \eqref{partiL}, where $\widetilde{g} $ is the function given by Lemma \ref{Le:g}. Since  
$$\begin{array}{lll}
D_k\psi (x) & =& \ds  \theta_R^2\sum_{i=1}^{n} \lambda_i  D_{i}\varphi  D_{ik}\widetilde{g} 
\\
\\
&& + \theta_R^2\sum_{i=1}^{n} \lambda_i D_{ik}\varphi D_i\widetilde{g} + 2 \theta_R D_k \theta_R \langle Q^{1/2}D \varphi, Q^{1/2}D \widetilde{g} \rangle ,
\end{array}$$
replacing in  \eqref{partiL} and recalling that $D \widetilde{g} = D g$ at $\partial O$ we get
$$\begin{array}{l}
\ds{\int_{\partial O } \theta_R^2 \frac{|\langle Q^{1/2}D \varphi, Q^{1/2}D g \rangle|^2}{|Q^{1/2}D g|} d\sigma = 2 \int_{O} \theta_R^2 {\mathcal L}\varphi \,  \langle Q^{1/2}D \varphi, Q^{1/2}D \widetilde{g} \rangle \, d\mu}
\\
\\
\ds{ +  \int_{O} \theta_R^2 \sum_{i, k=1}^{n}  \lambda_{k}  \lambda_{i}  ( D_{ik}\varphi D_k\varphi D_i\widetilde{g} + D_{ik}\widetilde{g} D_i\varphi D_k\varphi)
\,d\mu } 
\\
\\
\ds{ + 2 \int_{O} \theta_R  \langle Q^{1/2}D \varphi,  D  \theta_R \rangle  \langle Q^{1/2}D \varphi,  D \widetilde{g} \rangle  \, d\mu}
\\
\\
\ds{ =  2  \int_{O}  \theta_R^2  {\mathcal L}\varphi \,  \langle Q^{1/2}D \varphi, Q^{1/2}D \widetilde{g} \rangle d\mu 
+  \int_{O}  \theta_R^2 \langle Q D^2\varphi \, Q D \widetilde{g}, D \varphi\rangle d\mu }
\\
\\
\ds{+  \int_{O}   \langle  \theta_R^2 Q D^2 \widetilde{g} \, QD \varphi, D \varphi\rangle 
\, d\mu +  2  \int_{O}  \theta_R  \langle Q^{1/2}D \varphi,   Q^{1/2}D  \theta_R \rangle  \langle Q^{1/2}D \varphi,   Q^{1/2}D \widetilde{g} \rangle  \, d\mu} 
\\
\\
:= I_{1,R} + I_{2,R} + I_{3,R} +I_{4,R}.
\end{array}$$
The modulus of $I_{1,R}$  does not exceed  $\| \theta_R^2  {\mathcal L}\varphi\|_{L^2(O, \mu)}    \|\,|Q^{1/2}D \varphi|\,\|_{L^2(O, \mu)}  A$. Moreover, since 
$$ |\sum_{h, k=1}^{n}  \lambda_{k}  \lambda_{h} D_{hk}\varphi D_k\varphi D_h \widetilde{g} |  \leq  
(\sum_{h, k=1}^{n}  \lambda_{k}  \lambda_{h}  (D_{hk}\varphi )^2)^{1/2} |Q^{1/2}D \varphi| \, |Q^{1/2}D \widetilde{g}|,$$
we get 
$$|I_{2,R} | \leq \bigg(  \int_{O}  \theta_R^2 \sum_{h, k=1}^{n}  \lambda_{k}  \lambda_{h} (D_{hk}\varphi )^2d\mu  \bigg)^{1/2}  \|\,|Q^{1/2}D \varphi|\,\|_{L^2(O, \mu)} A. $$
Moreover, 
$$|I_{3,R}| \leq B  \|\,|Q^{1/2}D \varphi|\,\|_{L^2(O, \mu)} ^2, $$
and 
$$|I_{4,R}|\leq 2A   \|\,|Q^{1/2}D \varphi|\,\|_{L^2(O, \mu)} ^2  \|\,|Q^{1/2}D \theta_R|\,\|_{\infty}$$
Summing up, the statement  follows.  
\end{proof}

\begin{Remark}
\label{Rem:bounded}
If $O$ is bounded, we do not need to introduce the cutoff functions $\theta_R$ and  Lemma \ref{dernorm} may be restated as: for  every $\varphi  \in W^{2,2}( O, \mu)$ the function  $ \langle Q^{1/2}D \varphi, Q^{1/2}D g \rangle ^2/|Q^{1/2}D g|$ is in $L^1(\partial O)$ and 
\begin{equation}
\label{DerNorm}
\begin{array}{l}
\ds{ \int_{\partial  O} \frac{ \langle Q^{1/2}D \varphi, Q^{1/2}D g \rangle ^2}{|Q^{1/2}D g|}\, d\sigma }
\\
\\
\ds{ \leq   \bigg(2 \| {\mathcal L}\varphi \|_{L^2(O,\mu)} + \bigg( \int_{ O}  \mbox{\rm Tr}\;[(Q D^2\varphi)^2]d\mu \bigg)^{1/2}
\bigg)   \|\,|Q^{1/2}D \varphi |\,\|_{L^2(O,\mu)} A }
\\
\\
\ds{ + B  \|\,|Q^{1/2}D \varphi|\,\|_{L^2(O,\mu)} ^2.}
\end{array}
\end{equation}
 \end{Remark}
 
To get dimension free estimates 
for $u$, 
the last assumption we need  is that the function $h$ defined in \eqref{hn}
is bounded from above in  $ \partial O$. Then, we may state the main result of this section.

\begin{Theorem}
\label{stimen}
Let $g\in C^{2+\alpha}_{loc}(\R^n)$ satisfy \eqref{Ondelta}
and let $A$, $B$ be defined by \eqref{A,B}. Moreover  assume that $D g$ does not vanish at $\partial O$ and that 
\begin{equation}
\label{curvaturen}
C: =  \sup_{x\in \partial O}h(x) <\infty. 
\end{equation}
Then for every $f\in L^2(O, \mu)$ the weak solution $u$ of problem \eqref{Diri_n} belongs to $W^{2,2}(O, \mu)$ and 
\begin{equation}
\label{stimaX2n}
\int_O \mbox{\rm Tr}\;[(Q D^2u)^2]d\mu \leq \bigg[ 8  + 4\max\{C, 0\} \bigg( 2+\frac{2\sqrt{2}}{\sqrt{\lambda}}A +
\frac{1}{\lambda}|C|A^2 +\frac{B}{\lambda} \bigg) \bigg] \|f\|_{L^2(O,\mu)}^2.
\end{equation}
\end{Theorem}
\begin{proof}
It is sufficient to prove \eqref{stimaX2n} for $f\in C^{\infty}_{c}(O)$. Indeed, once  \eqref{stimaX2n} is established for  $f\in  C^{\infty}_{c}(O)$, it holds for $f\in L^2(O, \mu)$ since $C^{\infty}_{c}(O)$ is dense in $ L^2(O, \mu)$ by \cite[Lemma 2.1]{LMP}.

So, let $f\in C^{\infty}_{c}(O)$. If $O$ is bounded we use formula \eqref{e3.1bounded}, Lemma \ref{Le:identita}, and then   \eqref{stimaX1}(i), that give 
\begin{equation}
\label{parziale}
\begin{array}{lll}
\ds{ \frac{1}{2} \int_O \mbox{\rm Tr}\;[(Q D^2u)^2]d\mu } & \leq & 
\ds{ 2 \int_O (f-\lambda u)f\,d\mu + 
\frac{1}{2} \int_{\partial O} \frac{ \langle Q^{1/2}D u,   Q^{1/2}D g\rangle ^2}{|Q^{1/2}D g|}  \,h\,d\sigma}
\\
\\
& \leq & \ds{ 4\|  f \|_{L^2(O,\mu)}^2  + 
\frac{C}{2} \int_{\partial O} \frac{ \langle Q^{1/2}D u,   Q^{1/2}D g\rangle ^2}{|Q^{1/2}D g|}  \,d\sigma}.
 \end{array}
 \end{equation}
If $C\leq 0$,   \eqref{stimaX2n} follows. If $C>0$ we use Remark \ref{Rem:bounded}  that gives
$$ \begin{array}{l}
\ds{\int_{\partial O} \frac{ \langle Q^{1/2}D u,   Q^{1/2}D g\rangle ^2}{|Q^{1/2}D g|}  \,d\sigma}
\\
\\
\ds{ \leq   \bigg(2 \|\lambda u - f \|_{L^2(O,\mu)} + \bigg( \int_{ O}  \mbox{\rm Tr}\;[(Q D^2u)^2]d\mu \bigg)^{1/2}
\bigg)   \|\,|Q^{1/2}D u |\,\|_{L^2(O,\mu)} A }
\\
\\
\ds{ + B    \|\,|Q^{1/2}D u|\,\|_{L^2(O,\mu)}^2.}
\end{array}
$$
Using now estimates  \eqref{stimaX1} and $ab\leq (a^2+b^2)/2$ we get 
\begin{equation}
\label{parziale2}
 \begin{array}{l}
\ds{\int_{\partial O} \frac{ \langle Q^{1/2}D u,   Q^{1/2}D g\rangle ^2}{|Q^{1/2}D g|}  \,d\sigma
\leq \frac{4A\sqrt{2}}{\sqrt{\lambda}} \|f\|_{L^2(O,\mu)}^2}
\\
\\
\ds{ + \frac{1}{2C} \int_O \mbox{\rm Tr}\;[(Q D^2u)^2]d\mu + \frac{CA^2 + 2B}{\lambda}  \|f\|_{L^2(O,\mu)}^2} 
\end{array}
\end{equation}
and replacing in \eqref{parziale}, estimate \eqref{stimaX2n} follows. 

If $O$ is unbounded we need to cut off.  For $\eps \in (0, 1)$ and $R$ large we use \eqref{e3.1} instead of 
 \eqref{e3.1bounded}. Then, Lemma \ref{Le:identita} and  estimate   \eqref{stimaX1}(i) yield
\begin{equation}
\label{parzialeR}
\begin{array}{l}
\ds \bigg(  \frac12 -\frac{\eps}{2} \bigg) \int_O \theta_R^2\mbox{\rm Tr}\;[(Q D^2u)^2]d\mu   
 \\
 \\
\ds \leq  
\bigg(4 + \frac{\eps}{\lambda} \bigg) \|  f \|_{L^2(O,\mu)}^2 +  \frac12\int_{\partial O} \theta_R^2  \frac{D^2u(QD g,  QD u  \rangle }{|Q^{1/2}D g| }\, d\sigma
\\
\\
\ds \leq   
\bigg(4 + \frac{\eps}{\lambda} \bigg)\|  f \|_{L^2(O,\mu)}^2  + 
\frac{C}{2} \int_{\partial O}\theta_R^2  \frac{ \langle Q^{1/2}D u,   Q^{1/2}D g\rangle ^2}{|Q^{1/2}D g|}  \,d\sigma.
 \end{array}
 \end{equation}
If $C\leq 0$, letting $R\to \infty$ and then $\eps\to 0$, estimate \eqref{stimaX2n} follows. If $C>0$ we use Lemma \ref{dernorm}, that gives
$$ \begin{array}{l}
\ds{\int_{\partial O}\theta_R^2 \frac{ \langle Q^{1/2}D u,   Q^{1/2}D g\rangle ^2}{|Q^{1/2}D g|}  \,d\sigma}
\\
\\
\ds{ \leq   \bigg(2 \|\lambda u - f \|_{L^2(O,\mu)} + ( \int_{ O} \theta_R^2 \mbox{\rm Tr}\;[(Q D^2u)^2]d\mu )^{1/2}
\bigg)   \|\,|Q^{1/2}D u |\,\|_{L^2(O,\mu)} A }
\\
\\
\ds{ + \bigg( \frac{4A}{R}\,\lambda_1^{1/2} + B  \bigg)   \|\,|Q^{1/2}D u|\,\|_{L^2(O,\mu)}^2.}
\end{array} $$
so that \eqref{parziale2} is replaced by 
$$ \begin{array}{l}
\ds{\int_{\partial O} \theta_R^2 \frac{ \langle Q^{1/2}D u,   Q^{1/2}D g\rangle ^2}{|Q^{1/2}D g|}  \,d\sigma
\leq \frac{4A\sqrt{2}}{\sqrt{\lambda}} \|f\|_{L^2(O,\mu)}^2}
\\
\\
\ds{ + \frac{1}{2C} \int_O \mbox{\rm Tr}\;[(Q D^2u)^2]d\mu + \bigg(  \frac{CA^2 + 2B }{\lambda} + \frac{8A}{R\lambda} 
\lambda_1^{1/2} \bigg)   \|f\|_{L^2(O,\mu)}^2}.  
\end{array}$$
Replacing in  \eqref{parzialeR}, letting $R\to \infty$ and then $\eps \to 0$, estimate \eqref{stimaX2n} follows.   
\end{proof}

\begin{Corollary}
\label{cor:stimen}
Under the   assumptions of Theorem \ref{stimen}, for each $\lambda >0$  there exists $K=K(\lambda)>0$, independent of $n$, such that  for each $f\in L^2(O, \mu)$ the weak solution $u$ of problem \eqref{Diri_n} satisfies
$$\|u\|_{W^{2,2}(O, \mu)} \leq K \|f\|_{L^2(O, \mu)}.$$
\end{Corollary}
\begin{proof} It is sufficient to put together estimate \eqref{stimaX2n} and   estimates \eqref{stimaX1}  for $u$ and its first order derivatives.
\end{proof}


\section{Approximation by cylindrical functions}


Let us go back to infinite dimensions. We recall that $P_n$ is the orthogonal projection on the linear span of the first $n$ elements of the basis, 
$$P_nx := \sum_{k=1}^{n}  \langle x, e_k\rangle e_k  = \sum_{k=1}^{n} x_ke_k, \quad x\in X, $$
and that
$ \O = \{x\in X:\; G(x) <0\}$. We assume   that  $G\in   C^{2+\alpha}_{H,loc}(X )$ for some $\alpha >0$,  and that $D_H G$ does not vanish at $G^{-1}(0)$. 
%
%
To avoid pathologies we make a slightly stronger assumption,  that there exists $k_0\in \N$ such that 
\begin{equation}
\label{eq:Pn0}
P_{k_0}D_H G(x)\neq 0, \quad \forall x\in G^{-1}(0). 
\end{equation}

Moreover (as in finite dimensions) we assume that there is $\delta >0$ such that 
\begin{equation}
\label{eq:A}
\sup_{-\delta \leq G(x)\leq \delta }|D_HG(x) |_H =  \sup_{-\delta \leq G(x)\leq \delta } 
 \bigg(\sum_{ k=1}^{\infty} \lambda_k(D_{ k}G(x) )^2  \bigg)^{1/2}:= a<\infty ;
\end{equation}
\begin{equation}
\label{eq:B}
\sup_{-\delta \leq G(x)\leq \delta } \| D^2_HG(x) \|_{HS} =
  \sup_{-\delta \leq G(x)\leq \delta } \bigg(\sum_{h,k=1}^{\infty} \lambda_h(D_{hk}G(x) )^2\lambda_k \bigg)^{1/2}  := b <\infty .
\end{equation}

The aim of this section is to prove that for each $F\in L^2(\O, \mu)$ the weak solution $U$ to  \eqref{e1e} belongs to $W^{2,2}(\O, \mu)$, and the mapping $L^2(\O, \mu)\mapsto W^{2,2}(\O, \mu)$, $F\mapsto U$ is continuous.  Since the null extension $\widetilde{F}$ of $F$ to $X$ is in $L^2(X, \mu)$, and $C_b(X)$ is dense in $L^2(X, \mu)$, we may assume from the very beginning  that $F$ is defined in the whole of $X$ and that  it belongs to   $C_b(X)$. In this case, we shall approach $U$ by the solutions $U_n$ to problems 
\begin{equation}
\label{Pn}
\left\{\begin{array}{l}
\lambda U_n-\mathcal L U_n= F\circ P_n,\quad\mbox{\rm in}\;  \mathcal O_n,\\
\\
U_n=0,\quad\mbox{\rm on}\;\partial\mathcal O_n,
\end{array}\right.
\end{equation}
where  $\O_n$ is the cylindrical set defined by 
$$\O_n := \{ x\in X:\; G(P_n(x)) <0\}. $$
Define
$${\mathcal H}(x):= -2\frac{L^XG(x)}{ |D_HG(x) |^2} + \frac{ D^2_HG(x)( D_H G(x),  D_HG(x)) }{ |D_HG(x) |^4}, \quad x\in \partial \O, $$
and 
$${\mathcal H}_n(x):= -2\frac{L^XG_n(x)}{ |D_HG_n(x) |^2} + \frac{ D^2_HG_n(x) (D_H G_n(x),  D_HG_n(x)) }{ |D_HG_n(x) |^4}, \quad x\in \partial \O_n $$
(note that  $D_HG_n (x) = P_n D_HG(P_nx)\neq 0$ for each  $x\in \partial \O_n$ if $n\geq k_0$, by \eqref{eq:Pn0}). 
Our last  assumption is 
\begin{equation}
\label{curvatura}
 \exists n_0\in \N, \; c>0 :\; \;  \sup_{x\in \partial \O_n}{\mathcal H}_n(x)  \leq c, \quad n\geq n_0. \end{equation}

It is possible to see that for each $x\in \partial \O$, ${\mathcal H}_n(x)$ converges   to ${\mathcal H}(x)$. Then, assumption \eqref{curvatura} implies that ${\mathcal H}$ is bounded from above in $\partial \O$. 

Problem \eqref{Pn} lives in $P_n(X)$, so that  it may be identified with a problem in $\R^n$.

\begin{Lemma}
\label{Le:verifica}
Let $G\in C^{2+\alpha}_{H, loc}(X ) $  satisfy \eqref{eq:Pn0}, \eqref{eq:A}, \eqref{eq:B}, \eqref{curvatura}. Define 
\begin{equation}
\label{gn}
g_n(\xi)  = G(\sum_{k\leq n}\xi_ke_k), \quad \xi\in \R^n, 
\end{equation}
and 
\begin{equation}
\label{On}
O_n := \{ \xi\in \R^n:\;g_n(\xi) < 0\}. 
\end{equation}
Then for  $n$ large enough  the functions $g_n $ and the open sets $O_n$ satisfy the assumptions of Theorem \ref{stimen}, with constants $A$, $B$, $C$ independent of $n$.   
\end{Lemma}
\begin{proof} The proof is a simple check.  It is clear that $g_n\in C^{2+\alpha}_{loc}(\R^n)$.  Moreover, 
$$D_ig_n (\xi) = D_iG(\sum_{i=1}^n \xi_ie_i), \quad D_{ij}g_n (\xi) = D_{ij}G(\sum_{i=1}^n \xi_ie_i), \qquad i,j=1, \ldots, n.$$
If  $\xi \in \partial O_n$, then $G(\sum_{i=1}^n \xi_ie_i)=0$ so that the gradient of $g_n$ at $\xi$ does not vanish by \eqref{eq:Pn0}, for $n\geq k_0$. 
Moreover, since $-\delta \leq g_n(\xi) \leq \delta$ iff $-\delta \leq G( \sum_{k\leq n}\xi_ke_k) \leq \delta$, we get
$$\sup_{-\delta \leq g_n(\xi) \leq \delta } |Q_n^{1/2}D g_n(\xi)| \leq  a, \quad \sup_{-\delta \leq g_n(\xi) \leq \delta} \|Q_n^{1/2}D^2g_n(\xi)Q_n^{1/2}\|_{HS} \leq b.$$
So, the functions $g_n$ satisfy \eqref{Ondelta} with constants $a$, $b$. Assumption \eqref{curvatura} implies that they satisfy \eqref{curvaturen} with supremum  $C\leq c$, for $n\geq n_0$.   \end{proof}

Set now $F_n  :=F\circ P_n$, and accordingly
\begin{equation}
\label{fn}
f_n(\xi) := F (\sum_{k\leq n}\xi_ke_k) , \quad \xi\in \R^n.
\end{equation}
As it is easy to see, $F_n$ converges to $F$ pointwise and in $L^2(X, \mu)$.
Corollary \ref{cor:stimen} guarantees that  for every $\lambda >0$   the weak solution $u_n $ of problem 
\begin{equation}
\label{probleman}
\left\{ \begin{array}{l}
\lambda u_n  - {\mathcal L}_n u_n  = f_n, \quad \mbox{\rm in}\; O_n, 
\\
\\
u _n= 0, \quad \mbox{\rm in}\;  \partial O_n
\end{array}\right. 
\end{equation}
belongs to $W^{2,2}(O_n, \mu_n)$ for large $n$, and satisfies 
\begin{equation}
\label{maggJFA}
\|u_n\|_{W^{2,2}(O_n, \mu_n)} \leq K\|f_n\|_{L^2(O_n, \mu_n)} = K\|F_n\|_{L^2(O_n, \mu)}\leq K  \|F\|_{\infty} 
\end{equation}
with $K=K(\lambda)>0$ independent of $n$. Here ${\mathcal L}_n$ is the Ornstein--Uhlenbeck operator defined in  \eqref{OUinRn}. 

By the definition of weak solution we get  
\begin{equation}
\label{Un}
U_n(x) =  \widetilde{u}_n(x_1, \ldots, x_n), \quad x\in X. 
\end{equation}
where $ \widetilde{u}_n$ the null extension of $u_n$ outside $O_n$. 
 So, $\|U_n\|_{ W^{1,2}(X, \mu)}$ is bounded by a constant independent of $n$, and a subsequence converges weakly to a limit function $V$ in $W^{1,2}(X, \mu)$. 

Now, our aim is 
\begin{itemize}
\item[(i)] to show that $V_{|\O}\in \oo{W}^{1,2}(\O, \mu)\cap W^{2,2}(\O, \mu)$, 
\item[(ii)]  to show that
$V_{|\O}$ is a weak solution to  \eqref{e1e} (so that, $V_{|\O}=U$).
\end{itemize}

The following lemma will be used twice. 

\begin{Lemma}
\label{Le:superfici}
Let $\nu $ be a Radon nondegenerate Gaussian measure in a Banach space $Y$, and let $\Phi: Y\mapsto \R$ be differentiable along the Cameron--Martin space $H$,  such that $D_H\Phi$ is continuous, and $D_H\Phi(y)\neq 0$  for each $y\in \Phi^{-1}(0)$. Then $\nu( \Phi^{-1}(0) ) =0$. 
\end{Lemma}  
\begin{proof}
Since $D_H\varphi \neq 0$ at $ \Phi^{-1}(0)$ and $D_H\Phi$ is continuous, there exists an open neighborhood $  A$ of $ \Phi^{-1}(0)$ such that $D_H\varphi(y)\neq 0$ for each $y\in \overline{  A}$. 
Let us consider the  measure $\widetilde{\nu}$ defined by 
$$\widetilde{\nu}( B) := \nu({  A} \cap B), \quad B\in {\mathcal B}(Y),$$
which is still a Radon measure   in $Y$, and whose support is $\overline{A}$.   $\nu$ is differentiable, hence continuous,  in the directions of its Cameron-Martin space (\cite[Prop. 5.1.6]{Bo}), 
and so is $\widetilde{\nu}$. The set $C:=\{ y\in Y:\;  D_H\Phi(y)=0\}$  does not intersect $\overline{A}$, so that $\widetilde{\nu}(C)=0$. By \cite[Thm. 9.2.4]{Boga2}, the measure $\widetilde{\nu} \circ \Phi^{-1}$ is absolutely continuous with respect to the Lebesgue measure. In particular, $0=\widetilde{\nu}( \Phi^{-1}(0))= \nu( \Phi^{-1}(0))$. 
\end{proof}

To begin with, we prove the regularity statement (i).

\begin{Proposition}
\label{Pr:V}
Let $V$ be any weak limit of a subsequence of $U_n$ in $W^{1,2}(X, \mu)$. Then   the restriction of $V$ to $\O$ belongs to  $W^{2,2}(\O, \mu)\cap \oo{W}^{1,2}(\O, \mu)$, and there is $N>0$, independent of $F$, such that 
\begin{equation}
\label{stimaV}
\|V_{|\O}\|_{W^{2,2}(\O, \mu)} \leq N \|F\|_{L^2(X, \mu)}. 
\end{equation}
\end{Proposition}
\begin{proof}
First of all we prove that $V$ vanishes a.e. in $\O^c$. To avoid sub-subindices, we shall assume that $U_n$ converges weakly to $V$. Since $U_n$ vanishes outside $\O_n$, we have
\begin{equation}
\label{V^2}
\int_{\O^c} V^2d\mu = \lim_{n\to \infty} \int_{\O^c} U_nV\, d\mu 
=  \lim_{n\to \infty} \int_{\O_n \cap{\mathcal O}^c} U_nV\, d\mu .
\end{equation}
Let us prove that $ \int_{\O_n \cap{\mathcal O}^c} U_nV\, d\mu $ goes to $0$ as $n\to \infty$. 

Lemma \ref{Le:superfici} implies  that  $\mu(\partial \O) =0$.  
We claim that  $\mu(\O_n \cap \O^c) = \mu(\O_n \cap \overline{\O}^c)$ vanishes as $n\to \infty$. Indeed, if  $x\in \overline{\O}^c$ then  $G(x) > 0$, hence eventually  $G(P_nx)>0$, that is  $x\notin \overline{\O}_n$. Therefore,  $\lim_{n\to \infty}\one_{\overline{\O}_n \cap \overline{\O}^c} =0$. By dominated convergence, 
$\mu ( \O_n \cap \overline{\O}^c) = \mu (\overline{\O}_n \cap \overline{\O}^c)= \int_X \one_{\overline{\O}_n \cap \overline{\O}^c}d\mu $ vanishes as $n\to \infty$. 

For each $n\in \N$ let us estimate:
$$ |\int_{\O_n \cap \O ^c} U_nV\, d\mu |\leq \|U_n\|_{L^2(\O_n \cap \O^c)}
\bigg(\int_{\O_n \cap \O^c}V^2 d\mu\bigg)^{1/2}$$
where $\|U_n\|_{L^2(\O_n \cap \O^c)}\leq \|U_n\|_{L^2(X, \mu)}\leq const.$ and 
$\int_{\O_n \cap \O^c}V^2 d\mu$ vanishes as $n\to \infty$ since $V^2\in L^1(X, \mu)$ and $\mu(\O_n \cap \O^c)$ vanishes as $n\to \infty$. By \eqref{V^2}, $\int_{\O^c} V^2d\mu =0$ so that 
  $V$ vanishes a.e. on $\O^c$. 

\vspace{3mm}
Now, we prove that $V$ has second order weak derivatives in $\O$. 
By   \eqref{maggJFA},   the restrictions of  $U_n$ to ${\mathcal O}_n$ belong to  $W^{2,2}({\mathcal O}_n, \mu)$ and 
$$
\int_{{\mathcal O}_n} \sum_{h,k}\lambda_h(D_{hk}U_n)^2\lambda_k \,d\mu \leq     K^2 \|F_n\|_{L^2(X, \mu)}^2, \quad n\geq n_0.
$$
Therefore, fixed any $i,j\in \N$, the functions $ D_{ij}U_n \one_{\O \cap \O_n}$ are bounded in $L^2(\O, \mu)$ by a constant independent of $n$. Then, there is a subsequence that converges weakly to a function $\psi_{ij} \in L^2(\O, \mu)$, and 
\begin{equation}
\label{deb}
\int_{{\mathcal O} } \sum_{h,k} \lambda_h \psi_{hk} ^2 \lambda_k\,d\mu \leq K^2 \lim\sup_{n\to \infty}\|F_n \|_{L^2(X, \mu)}^2 = K^2  \|F \|_{L^2(X, \mu)}^2. 
\end{equation}
Let us prove that $\psi_{ij} $ is the weak derivative of $D_jV$ in the direction $e_i$. For each $\Phi \in C^1_b(\O )$ that vanishes at $\partial \O$, the function $\Phi\circ P_n$ belongs to $  C^1_b(\O _n )$ and vanishes at $\partial \O_n$. So, 
\begin{equation}
\label{derdeboli}
\int_{\O_n} D_{ij}U_n\, \Phi\circ P_n\,d\mu =  \int_{\O_n} D_jU_n \bigg( -D_i(\Phi\circ P_n)+ \frac{x_i}{\lambda_i}\,\Phi\bigg) d\mu, \quad n\in \N. 
\end{equation}
The integral in the left hand side may be split  as 
$$\begin{array}{l}
\ds \int_{\O_n\setminus \O}D_{ij}U_n( \Phi\circ P_n)\,d\mu +  \int_{\O} D_{ij}U_n \, \one_{\O \cap \O_n}\Phi \, d\mu + 
\\
\\
\ds +  \int_{\O} D_{ij}U_n \, \one_{\O \cap \O_n}((\Phi\circ P_n)- \Phi) d\mu := I_{1,n}+I_{2,n} +I_{3,n},
\end{array}$$
where 
$$|I_{1,n}|\leq \|D_{ij}U_n\|_{L^2(\O_n, \mu)}\mu (\O_n\setminus\O)^{1/2}\|\Phi\|_{\infty}$$
 vanishes as $n\to \infty$, 
$I_{2,n}$ goes to $ \int_{\O} \psi_{ij} \Phi \, d\mu$ along the converging subsequence, 
$$|I_{3,n}|\leq  \|D_{ij}U_n\|_{L^2(\O_n, \mu)}\bigg(\int_{\O}(\Phi\circ P_n- \Phi)^2 d\mu \bigg)^{1/2}$$
vanishes as $n\to \infty$ by dominated convergence (recall that $\Phi$ is continuous, so that $\Phi\circ P_n$ converges pointwise to $\Phi$). 
 
Hence, the left-hand side of \eqref{derdeboli} goes to $ \int_{\O} \psi_{ij} \Phi \, d\mu$, at least along the converging subsequence. Concerning the  right-hand side, since   $U_n$ converges weakly in $W^{1,2}(X, \mu)$ to $V$ and  the $W^{1,2}$-norm of $U _n$ is bounded by a constant independent of $n$, then $D_jU_n \, \one_{\O \cap \O_n}$ converges weakly in $L^2(X, \mu)$ to $D_jV  \, \one_{\O}$. 
Therefore, splitting the  right-hand side  as
$$\begin{array}{l}
\ds
\int_{\O_n\setminus \O} D_jU_n \bigg(-D_i(\Phi\circ P_n)+ \frac{x_i}{\lambda_i} \Phi\circ P_n\bigg) d\mu + \int_{\O } D_jU_n \, \one_{\O \cap \O_n} \bigg(- D_i\Phi + \frac{x_i}{\lambda_i} \Phi\bigg)\,d\mu
\\
\\
\ds + 
 \int_{\O } D_jU_n \, \one_{\O \cap \O_n} \bigg(-D_i(\Phi\circ P_n)+ D_i\Phi + \frac{x_i}{\lambda_i}(\Phi\circ P_n-\Phi) \bigg)d\mu 
 \end{array}$$
and arguing as above, we see that the right-hand side  of \eqref{derdeboli} goes to $ \int_{\O } (-D_jVD_i\Phi +x_i \Phi/\lambda_i ) d\mu$
as $n\to \infty$. So, we have 
$$\int_{\O } \psi_{ij}  \Phi \,d\mu = \int_{\O } D_jV\bigg(-D_i \Phi +  \frac{x_i}{\lambda_i} \Phi\bigg) d\mu$$
which means that  the weak derivative $D_{ij}V$ in $\O$ exists and coincides with $\psi_{ij}$. 
Now \eqref{deb} implies that $V_{|\O}\in W^{2,2}(\O, \mu)$.  

Estimate \eqref{stimaV} follows now easily. Indeed, estimate \eqref{stimaX1} implies 
$\|U_n\|_{W^{1,2}(X, \mu)}^2 \leq  (\lambda^{-2}+ 2\lambda^{-1})\|F_n\|_{L^2(X, \mu)}^2$ for $n\geq n_0$, that in its turn implies $\|V\|_{W^{1,2}(\O, \mu)}  = \|V\|_{W^{1,2}(X, \mu)}     \leq  C(\lambda)$ $\|F\|_{L^2(X, \mu)} $.

Estimate \eqref{deb} and the equalities $D_{ij}V=\psi_{ij}$ in $\O$ imply $\int_{{\mathcal O} } \sum_{h,k} \lambda_h (D_{hk}V )^2 \lambda_k\,d\mu$ $ \leq  K^2\|F \|_{L^2(X, \mu)}^2$  for $n\geq n_0$, and  \eqref{stimaV} follows. 
\end{proof}

To identify  any weak limit $V$ with the weak solution $U$ to  \eqref{e1e}  we use the representation formula \eqref{1.9w}.  
 
\begin{Proposition}
\label{Pr:convessi}
For each $F\in C_b(X)$ the sequence  $(U_ {n| \O})$ converges pointwise (hence, in $L^2(\O, \mu)$) to the weak solution $U$ of  \eqref{e1e}.  Consequently, if a subsequence $U_{n_k}$ converges weakly to $V$ in $L^2(X, \mu)$, then $V_{|\O} = U$. 
\end{Proposition}
\begin{proof} 
For each  $x\in X$ let us consider  the solution $X(t,x)$  to  \eqref{e1.6}. For any $t_0>0$, the restriction of $X(\cdot, x)$ to $[0, t_0]$  is a Gaussian random variable with values in $C([0,t_0];X)$ (see the Appendix), as well as the restriction of $X_n:=P_nX$.  Denoting by $ \mu_{t_0,x}$ the law of $X(\cdot, x)$ in $ C([0,t_0];X)$ and using  \eqref{e1.7} and  \eqref{e1.8}, if $x\in \O$ we have
$$ (T^{\O}(t_0) F_{|\O})(x) =   \int_{\{\omega:\, X(s,x)(\omega)\in \overline{\O}, \,\forall s\in [0,t_0]\} }F(X(t_0,x))d\P 
= \int_{\Lambda_{t_0}} F(\eta (t_0)) \mu_{t_0,x}(d\eta)$$ 
where
\begin{equation}
\label{e5}
\Lambda_{t_0}:= \{\eta\in C([0,t_0];X):\;\sup_{s\in[0,t_0]}G(\eta(s))\le 0 \}.
\end{equation}

Fix $n\in \N$. The representation formula  \eqref{e1.8} yields 
$T^{\O_n}(t)\Phi(y)= \E[\Phi(X(t,y))\one_{\tau^n_y \ge t}]$, for each $t>0$, $\Phi\in C_b(\overline{\O_n})$ and $y\in \O_n$, where $\tau_y^n$ is the exit time of $X(\cdot, y)$ from $\overline{\O_n}$. For each $x\in \O$, $P_nx\in \O_n$, and 
$X(t, P_nx) = P_nX(t,x)$. So,  taking $t=t_0$, $\Phi =  F\circ P_{n|\O_n}$, $y=P_nx$ we have
$$\begin{array}{l}
(T^{\O_n}(t_0) F\circ P_{n|\O_n})(P_nx) =  \E[F(P_nX(t,x))\one_{\tau_{P_nx} \ge t_0}]
\\
\\
=\ds    \int_{\{\omega:\, P_nX(s,x)(\omega)\in \overline{\O_n}, \,\forall s\in [0,t_0]\} }F(P_nX(t_0,x))d\P 
= \int_{\Lambda_{t_0}} F(\eta (t_0)) \mu_{t_0,x}^{(n)}(d\eta)
\end{array}$$
where $ \mu_{t_0,x}^{(n)}$ is the law of $P_nX(\cdot, x)$ in $ C([0,t_0];X)$. Our aim is to show that for each $x\in \O$
and for each $t_0>0$ we have 
\begin{equation}
\label{convT(t)}
\lim_{n\to \infty} (T^{\O_n}(t_0) F\circ P_{n|\O_n})(P_nx) = (T^{\O}(t_0) F_{|\O})(x) .
\end{equation}
Once \eqref{convT(t)} is established, since $|T^{\O_n}(t_0) F\circ P_{n|\O_n}(P_nx) |\leq \|F\|_{\infty}$, 
  for each $\lambda >0$ by dominated convergence we obtain 
$$\lim_{n\to \infty}  \int_{0}^{+\infty} e^{-\lambda t}(T^{\O_n}(t) F\circ P_{n|\O_n})(P_nx)dt =  \int_{0}^{+\infty} e^{-\lambda t}(T^{\O}(t) F_{|\O})( x)dt $$
On the other hand, for $x\in \O $ we have $P_n x\in \O_n$, hence 
$\int_{0}^{+\infty} e^{-\lambda t}(T^{\O_n}(t) F\circ P_{n|\O_n})(P_nx)dt = U_n(P_nx) = U_n(x)$, while  $\int_{0}^{+\infty} e^{-\lambda t}(T^{\O}(t) F_{\O})(x)dt = U(x)$, by formula \eqref{1.9w}.  So,  $(U_ {n| \O})$ converges pointwise (hence, in $L^2(\O, \mu)$) to $(U_ {| \O})$. Consequently, if  any subsequence $U_{n_k}$ converges weakly to $V$ in $L^2(X, \mu)$, then $V_{|\O} = U$.
 
Formula \eqref{convT(t)} will be proved in three steps: first, we prove that $ \mu_{t_0,x}^{(n)}$ converges weakly to $ \mu_{t_0,x}$, then we prove that $ \mu_{t_0,x}(\partial \Lambda_{t_0})=0$, and eventually we conclude. 

\vspace{3mm}

\noindent {\em First step.} 
Let us prove that $ \mu_{t_0,x}^{(n)}$ converges weakly to $ \mu_{t_0,x}$. To this aim it is enough to show that 
\begin{equation}
\label{e10}
\lim_{n\to\infty}\E \bigg(\sup_{s\in[0,t_0]}|X(s,x)-X_n(s,P_nx)|^2 \bigg)=0.
\end{equation}
Setting $\varphi(t) := e^{-t/2}x$ and $\varphi_n(t) := e^{-t/2}P_nx$, we have $X(\cdot,x)(\omega)-\varphi =    \Gamma 
(\sqrt{Q}W(\cdot )(\omega)$ and $X_n(\cdot,P_nx)(\omega)-\varphi_n =   \Gamma 
((I-P_n)\sqrt{Q}W(\cdot )(\omega)$, where $\Gamma$ is the linear bounded operator in $C([0,t_0];X))$ defined in the Appendix (see \eqref{e3}). 
Then, 
$$\begin{array}{l} 
\sup_{s\in[0,t_0]}\|X(s,x)(\omega)-X_n(s,P_nx)(\omega)\|^2 \leq 
\\
\\
2(\|\Gamma\|_{ {\mathcal L}(C([0,t_0];X))}^2
\sup_{s\in[0,t_0]} \|(I-P_n)\sqrt{Q}W(s )(\omega)\| + \sup_{s\in[0,t_0]}\| \varphi(s) - \varphi_n(s)\|^2)
\end{array}$$
so that, recalling that  $(I-P_n)\sqrt{Q}W(\cdot )$ is a martingale and that $\E \|(I-P_n)W(t_0)\|^2 = 4t_0^2
 \;\mbox{\rm Tr}\;[(I-P_n) Q]$, 
$$\begin{array}{l}
 \E\left(\sup_{s\in[0,t_0]}\|X(s,x)-X_n(s,P_nx)\|^2 \right)
\leq 
\\
\\
2\|\Gamma\|_{ {\mathcal L}(C([0,t_0];X))}^2\cdot 4t_0^2
 \;\mbox{\rm Tr}\;[(I-P_n) Q] + 2 \|(I-P_n)x\|^2
 \end{array}$$
that vanishes as $n\to \infty$, and \eqref{e10} follows.

\vspace{3mm}

\noindent {\em Second step.}
Let us prove that $ \mu_{t_0,x}(\partial \Lambda_{t_0})=0$. We have
$$\Lambda_{t_0} = \bigcap_{s\in \Q \cap [0, t_0]}  \{\eta\in C([0,t_0];X):\; G(\eta(s))\le 0 \}, $$
and its complement in $C([0,t_0];X)$ is the set 
$$ \bigcup_{s\in \Q \cap [0, t_0]}  \{\eta\in C([0,t_0];X):\; G(\eta(s))> 0 \}. $$
Therefore, 
$$\begin{array}{lll}
\partial \Lambda_{t_0} & \subset &  \bigcup_{s\in \Q \cap [0, t_0]}\partial   \{\eta\in C([0,t_0];X):\; G(\eta(s))> 0 \}
\\
\\
&=  & \bigcup_{s\in \Q \cap [0, t_0]}   \{\eta\in C([0,t_0];X):\; G(\eta(s))= 0 \}.
\end{array}$$
So, it is enough to prove that  $\mu_{t_0,x} (\{\eta\in C([0,t_0];X):\; G(\eta(s))= 0 \})$  for each    $s\in [0, t_0]$. 

By Proposition \ref{App}, the Cameron--Martin space $\mathcal H$ of the measure $ \mu_{t_0,x}$ is contained in $C([0,t_0];H)$. Therefore, 
fixed any $s\in [0,t_0]$, the function 
$$\Phi: C([0,t_0];X)\mapsto \R, \quad \Phi(\eta) =  G(\eta(s)),$$
is continuously differentiable along $\mathcal H$, its  Frech\'et derivative at any $\eta_0\in C([0,t_0];X)$ is given by 
$$D_{\mathcal H}\Phi(\eta_0)(\eta) =  \langle D_HG(\eta_0(s)), \eta(s)\rangle , $$
and it does not vanish if $D_HG(\eta_0(s))\neq 0$. Since $D_H G\neq 0$ at $\partial \O $, then $D_{\mathcal H}\Phi(\eta_0)$ does not vanish at any $\eta_0\in \Phi^{-1}(0)$. Moreover, since $C([0,t_0];X)$ is separable, then $ \mu_{t_0,x}$ is a Radon   measure. Lemma \ref{Le:superfici} implies that $\mu_{t_0,x} \Phi^{-1}(0) =\mu_{t_0,x} (\{\eta\in C([0,t_0];X):\; G(\eta(s))= 0 \}) =0$, and the claim follows.

\vspace{3mm}

\noindent {\em Third step: conclusion.}
For each $F\in C_b(X)$ the set of discontinuities of the function $\eta \mapsto F(\eta(t_0))\one_{\Lambda_{t_0}}$ is contained in the boundary of $\Lambda_{t_0}$, whose measure vanishes, and therefore by a well known convergence theorem (e.g. \cite[Thm. 5.2(iii)]{Bill}) $\lim_{n\to \infty} \int_{\Lambda_{t_0}} F(\eta (t_0)) \mu_{t_0,x}^{(n)}(d\eta) = \int_{\Lambda_{t_0}} F(\eta (t_0)) \mu_{t_0,x}(d\eta)$, that is 
$$\lim_{n\to \infty} T^{\O_n}(t_0) F(P_nx) = T^{\O}(t_0) F(x).$$
\end{proof} 
 
We resume the results of this section in the next theorem. 
 
\begin{Theorem}
\label{Th:finale}
Let $G\in C^{2+\alpha}_{H,loc}(X)$ for some $\alpha >0$ satisfy \eqref{eq:Pn0},  \eqref{eq:A}, \eqref{eq:B}, \eqref{curvatura}, and let  $\O = \{x\in X:\; G(x) <0\}$. 
Then for each $F \in L^2(\O, \mu)$ the weak solution $U$ to \eqref{e1e} belongs to $W^{2,2}(\O, \mu)$, and the mapping $F\mapsto U$  is continuous  from $L^2(\O, \mu)$ to $W^{2,2}(\O, \mu)$. 
\end{Theorem}
\begin{proof} Let $F\in C_b(X)$ and let  $g_n$, $O_n$, $f_n$  be defined by \eqref{gn}, \eqref{On}, \eqref{fn}, respectively. By Lemma \ref{Le:verifica}, 
for $n$ large enough  the functions $g_n  $ satisfy the hypotheses of Theorem \ref{stimen}. Moreover, as already remarked   $\|f_n\|_{L^2(O_n, \mu_n)} $ is bounded by a constant independent of $n$. 
By Corollary \ref{cor:stimen}, the solutions 
$u_n$ of problems  \eqref{probleman} with  $f= f_n$ 
belong  to $\oo{W}^{1,2}(O_n, \mu_n) \cap W^{2,2}(O_n, \mu_n)$, and  their $W^{2,2}(O_n, \mu_n)$ norm is bounded by a constant independent of $n$. 
By \eqref{Un}, the sequence of the solutions $U_n$ to \eqref{Pn} is bounded in $W^{1,2}(X, \mu)$.  Then, a subsequence converges weakly to a limit function $V$ in $W^{1,2}(X, \mu)$, whose restriction to  $\O$ belongs to $\oo{W}^{1,2}(\O, \mu) \cap W^{2,2}(\O, \mu)$ by Proposition  \ref{Pr:V}. 
$V_{|\O}$ coincides with the weak solution $U$ to  \eqref{e1e} by Proposition   \ref{Pr:convessi}. 

Estimate  \eqref{stimaV}  implies that 
$\|U\|_{ W^{2,2}(\O, \mu)} \leq N\|F\|_{L^2(X, \mu)}$, with $N$ independent of $F$. Approximating the null extension $\widetilde{F}$ of any $F\in L^2(\O, \mu)$ by a sequence of functions in $C_b(X)$ yields the final statement. 
\end{proof}


\section{Examples}
\label{Examples}

\noindent {\em 1. Half--spaces.} The simplest examples of open sets satisfying our assumptions are half--spaces, 
$$\O = \{x\in X: \;\langle b, x\rangle <c\},  $$
where $b\in X\setminus \{0\}$ and $c\in \R$ are given. Indeed, the function $G(x):= \langle b, x\rangle -c$ is smooth and $ D G(x) = b $ at each $x$, so that $D_HG(x) = Q b$, $D^2_HG(x) = 0$, $L^X G = -G/2 -c/2$
 and assumptions \eqref{eq:Pn0}, \eqref{eq:A}, \eqref{eq:B}  are obviously satisfied. Since 
$$G_n(x) = \sum_{k=1}^{n} b_kx_k - c  , $$
for $n$ large enough ($n\geq \min\{k:\; b_k\neq 0\}$) $\O_n$ is a half--space in $X$, and 
$${\mathcal L}^XG_n(\xi) = -\frac{1}{2}  \sum_{k=1}^{n}  b_k x_k  = - G_n(x)/2 -c/2, \quad x\in X, $$
 $${\mathcal H}_n(x) = \frac{ \sum_{k=1}^{n}  b_k x_k }{ \sum_{k=1}^{n}\lambda_k  b_k^2} =  \frac{c}{ \sum_{k=1}^{n}\lambda_k  b_k^2}, \quad x \in \partial \O_n, $$
so that   the curvature condition \eqref{curvatura} holds. 
Then, all the assumptions of Theorem \ref{Th:finale} are satisfied. 

\vspace{3mm}

\noindent {\em 2. Graphs.} A natural generalization of half--spaces is the case where $\O$ is the region below the graph of a good function, $\O = \{ x\in X:\; x_k < \Phi(\widetilde{x}_k)\}$ for some $k\in \N$, where $\widetilde{x}_k = x - x_k e_k$ and $\Phi\in C^{2+\alpha}_{H,loc}(X)$. 
The function $G$ is now
$$G(x) = x_k -  \Phi(\widetilde{x}_k)= x_k -  \Phi(x - x_k e_k), \quad x\in X, $$
so that $|D_HG(x)|_H\geq \lambda_k^{1/2}$ at each $x$. Assumptions \eqref{eq:A}, \eqref{eq:B}  are satisfied if $|D_H\Phi|_H$ and $\|D^2_H\Phi\|_{HS}$  are bounded. 

For $n\geq k$ we have
$G_n(x) = x_k -   \Phi( P_nx - x_ke_k)$, 
and  $P_n D_HG  \neq 0$ at each $x$. Moreover, if $G_n(x) =0$ we have 
$$ 
{\mathcal H}_n(x) =  \frac{ \Phi + \sum_{h\leq n, h\neq k}(\lambda_h D_{hh} \Phi 
-x_hD_h \Phi )}{\lambda_k +  \sum_{h\leq n, h\neq k}\lambda_h(D_h \Phi )^2}
+ \frac{ \sum_{h, l\leq n, \,h\neq k, l\neq k}\lambda_h\lambda_l D_{hl} \Phi D_h \Phi D_l \Phi }{(\lambda_k +  \sum_{h\leq n, h\neq k}\lambda_h (D_h \Phi )^2)^2}
$$
where $\Phi$ and its derivatives are evaluated at $ P_nx - x_ke_k$. 
Obvious sufficient conditions for  \eqref{curvatura} hold are $\Phi \leq C$, $|\langle x, D\Phi(x)\rangle |\leq C$, $|D_{hl}\Phi|\leq C$,   for $h$, $l\neq k$, with $C$ independent of $h$ and $l$.

\vspace{3mm}

\noindent {\em 3. Spheres.}  Let $x^0\in X$, $r>0$ and consider the function $G(x) := \|x-x^0\|^2 - r^2$. We have
$$D G(x) = 2(x-x_0), \quad D^2G(x) = 2I, \quad L^XG(x) =  {\rm Tr}\,Q - \langle x-x^0, x\rangle  , $$
so that \eqref{eq:A},  \eqref{eq:B}  are   satisfied. Moreover, 
$$G_n(x) = \sum_{k=1}^{n} (x_k - x^0_k)^2 + \sum_{k=n+1}^{\infty} (x^0_k)^2 - r^2 $$
so that, for large $n$ (such that $\sum_{k=n+1}^{\infty} ( x^0_k)^2 < r^2$), $O_n$ is the open sphere centered at $\xi^0 = (x^0_1, \ldots , x^0_n)$ with radius $r_n = (r^2 - \sum_{k=n+1}^{\infty} (x^0_k)^2)^{1/2}$, $\O_n $ is the cylinder $O_n \times (I-P_n)(X)$, and \eqref{eq:Pn0} is satisfied. 
Moreover, $L^X G_n(x)  = \sum_{k=1}^{n}\lambda_k - \sum_{k=1}^{n} (x_k - x^0_k)x_k $ 
and
$${\mathcal H}_n(x) = \frac{1}{2 \|Q^{1/2}P_n(x - x^0)\|^2  }\bigg(\langle P_n(x-x^0),  x \rangle  - \sum_{k=1}^{n}\lambda_k + \frac{\|QP_n(x-x^0)\|^2 }{\|Q^{1/2}P_n(x- x^0)\|^2 } \bigg).$$
No upper bound independent of $n$ for $1/  \|Q^{1/2}P_n(x- x^0)\|^2$ on the   surface $\partial \O_n = \{x:\; G_n(x) =0\}$ is available: taking $x = (x^0_n + r_n)e_n$   we get $  1/  \|Q^{1/2}P_n(x - x^0)\|^2 = 1/ \lambda_nr_n^2$, that blows up as $n\to \infty$. However, condition \eqref{curvatura} is satisfied if 
\begin{equation}
\label{curv_sfere}
\langle P_n(x-x^0), x \rangle  - \sum_{k=1}^{n}\lambda_k + \frac{\|QP_n(x-x^0)\|^2}{\|Q^{1/2}P_n(x - x^0)\|^2} \leq 0, \quad \mbox{\rm for}\; \|P_n(x-x^0)\| = r_n, 
 \end{equation}
 for large $n$.  Using the obvious estimates  $|\langle P_n(x-x^0), x \rangle | \leq r_n( \|P_nx^0\|  + r_n)$ and 
$\|QP_n(x-x^0)\|^2  /\|Q^{1/2}P_n(x-x^0)\|^2  \leq \lambda_1$, we see that \eqref{curv_sfere} is eventually satisfied provided $r$ is small enough, and precisely
$$r(\|x^0\| + r) < \sum_{k=2}^{\infty} \lambda_k .$$
In this case, all the assumptions of Theorem \ref{Th:finale} hold. 

If $r$ is large, condition \eqref{curvatura} does not hold. Indeed, 
$$\sup  {\mathcal H}_n\geq {\mathcal H}_n(x_0 + r_ne_n) = \frac{1}{2\lambda_n r_n^2}\bigg( r_n^2 + r_n\langle x_0, e_n\rangle - \sum_{k=1}^{n-1}\lambda_k\bigg)$$
which blows up as $n\to \infty$ if $r^2> $ Tr$\,Q$. 

Note that if we change $G$ into $-G$, the functions ${\mathcal H}_n$ change sign. So, if $\O = \{ x\in X:\; \|x-x^0\|^2 > r^2\}$,  the curvature condition \eqref{curvatura} is eventually satisfied provided $r(\|x^0\| + r) > \sum_{k=1}^{\infty} \lambda_k$.

Similar considerations hold for ellipsoids of the type $\O = \{x\in X: \, \langle T(x-x^0), x-x^0\rangle < r^2\}$, with
$x_0\in X$,  $r>0$ and $T$ is a diagonal operator defined by  $Te_k = t_ke_k$, $k\in \N$, for a positive and bounded  sequence $(t_k)$. In this case condition \eqref{curv_sfere} is replaced by 
$$ \sum_{k=1}^{n}t_k(  x_k-x^0_k)  x_k  - \sum_{k=1}^{n}\lambda_k t_k + \frac{ \sum_{k=1}^{n}\lambda_k^2t_k^3(x_k-x^0_k) ^2}{ \sum_{k=1}^{n}\lambda_k t_k^2(x_k - x^0_k) ^2} \leq 0  \quad \mbox{\rm for}\;\sum_{k=1}^{n} t_k (x_k - x^0_k) ^2 = r^2, $$
which  is satisfied for large $n$ provided
$$r(\|T^{1/2}x^0\| + r) < \sum_{k\in \N, \,k\neq \overline{k}}^{\infty} \lambda_k t_k, $$
where $\overline{k}$ is any natural number  such that $\lambda_{ \overline{k}}t_{ \overline{k}}$ is the biggest eigenvalue of $QT$.

\vspace{3mm}

\noindent {\em 4. A specific example in $X=L^2((0,1), dx)$.}  Let 
$$X= L^2((0,1), dx), \quad \lambda_k = \frac{1}{\pi^2k^2}, \quad e_k(\xi) = \sqrt{2}\sin(k\pi\xi), \;0<\xi<1, \;k\in \N.$$
So, for any $f\in X$, $Qf =\varphi$ is the solution to 
$$\varphi'' = -f, \quad  \varphi(0)=\varphi(1) =0, $$
and
$$H= Q^{1/2}(X) = H^1_0(0,1), \quad |h|_H = \| h'\|_X. $$
We define a function $G:X\mapsto \R$ by
$$G(x) = \int_0^1 g(x(s))ds -r, \quad x\in X, $$
where $g:\R\mapsto \R$ is a $C^2$ function with bounded second order derivative, and $r$ is in the range of $g$. Moreover, we assume that there exists $C>0$ such that 
\begin{equation}
\label{g3}
 |g''(\xi) - g''(\eta)|\leq C|\xi-\eta|(|\xi|+|\eta|), \quad \xi, \;\eta \in \R . 
\end{equation}

Then, $G$ is well defined and differentiable at any $x\in X$, and
\begin{equation}
\label{g1}
DG(x) = g'\circ x, \quad x\in X.
\end{equation}
Unless $g'$ is linear, $DG:X\mapsto X$ is not differentiable. However, since $g''$ is bounded and \eqref{g3} holds,   $DG$ is differentiable at any point  along $C([0,1])$, in particular it is differentiable along $H$, and we have
\begin{equation}
\label{g2}D^2_HG(x)(h, k) = \int_0^1 g''(x(s))h(s)k(s)ds, \quad h, \;k\in H.
\end{equation}
In fact, \eqref{g3} implies that $D^2_HG$ is  locally Lipschitz continuous. Then,  our regularity assumptions on $G$ 
($G \in C^{2+\alpha}_{H, loc}(X)$) are satisfied.

\begin{Proposition}
\label{strano}
Let $g\in C^2 (\R)$ have bounded first and second order derivatives, satisfy \eqref{g3} and moreover
\begin{equation}
\label{g4}
|g'(\xi)|\geq a ,  \quad \xi g'(\xi) \leq \alpha g(\xi) + \beta, \qquad \xi \in \R ,
\end{equation}
for some $a>0$, $\alpha$, $\beta \in \R$. Then for every $r\in \R$, the function $G(x) = \int_0^1 g(x(s))ds -r$ satisfies all the assumptions of Theorem \ref{Th:finale}. 
\end{Proposition}
\begin{proof} Since $|g'(\xi)|\geq a$ for every $\xi$, then the range of $g$ is the whole $\R$ and hence the range of $x\mapsto \int_0^1 g(x(s))ds$ is the whole $\R$. 

Assumptions  \eqref{eq:A}  and  \eqref{eq:B} are obvioulsy satisfied since $g'$ and $g''$ are bounded. 
Let us check that \eqref{eq:Pn0} holds. 
Since
$$ \langle DG( x), e_1\rangle =\sqrt{2}  \int_0^1 g'( x(s))\sin(\pi s)ds $$
by \eqref{g1}, we have $| \langle DG( x), e_1\rangle | \geq  2\sqrt{2}a $ for each $x\in X$, so that \eqref{eq:Pn0} is satisfied with $k_0=1$, and recalling that $G_n(x) = G(P_nx)$ we get
$$
|D_H G_n(x)|_H^2 \geq 8\lambda_1a^2 = \frac{ 8a^2}{ \pi^2} , \quad x\in X.
$$
To estimate ${\mathcal H}_n(x)$, we compute $L^XG_n(x)$. Using  \eqref{g1} and  \eqref{g2} we obtain
$$2L^XG_n(x) =\int_0^1 g''(P_nx(s))f_n(s)ds - \int_0^1 g'(P_nx(s)) (P_nx)(s)ds $$
where
$$f_n(s) = 2  \sum_{k=1}^n \lambda_k( \sin(k\pi s))^2$$
converges uniformly to 
$$f(s) = 2  \sum_{k=1}^{\infty} \lambda_k( \sin(k\pi s))^2 = s-s^2  $$
in the interval $[0,1]$. Therefore, setting $\varphi_n(s)  := D_H G_n(x)(s)$, we have
$$\varphi_n(s)  = P_nQ( g'\circ P_nx) (s) =  \sqrt{2} \sum_{k=1}^{n} \lambda_k a_k \sin(k\pi s), $$
with $a_k = \langle g'\circ P_nx, e_k\rangle$, and 
$$\begin{array}{lll}
{\mathcal H}_n(x) & = & \ds \frac{1}{|D_H G_n(x)|_H^2}\bigg(-2LG_n(x) +\frac{D^2_HG (P_nx)(QP_nDG(P_nx), QP_nDG(P_nx))}{|D_H G_n(x)|_H^2}\bigg)
\\
\\& = & \ds \frac{1}{|\varphi_n|^2_H} \bigg( \int_0^1 g''(P_nx(s))( - f_n(s) + \frac{(\varphi_n(s))^2}{|\varphi_n|^2_H})ds
+ \int_0^1  g'(P_nx(s)) (P_nx)(s)ds \bigg). 
\end{array}$$
Note that 
$$(\varphi_n(s))^2 \leq 2 \bigg(\sum_{k=1}^{n} \lambda_k a_k^2 \bigg) \bigg(  \sum_{k=1}^{n} \lambda_k ( \sin(k\pi s))^2\bigg)$$
so that 
\begin{equation}
\label{g5} \frac{(\varphi_n(s))^2}{|\varphi_n|_H^2} =
\frac{(\varphi_n(s))^2}{\sum_{k=1}^{n} \lambda_k a_k^2} \leq f_n(s).
\end{equation}
Then, 
$$\begin{array}{l}
\ds \int_0^1 g''(P_nx(s))( - f_n(s) + \frac{(\varphi_n(s))^2}{|\varphi_n|^2_H})ds
\leq \int_{\{s: g''(P_nx(s))<0\}} - g''(P_nx(s))  f_n(s)ds
\\
\\
\ds \leq \max\{ -\inf g'', 0\} \int_0^1f_n(s)ds \leq \frac{1}{6}\max\{ -\inf g'', 0\}. 
\end{array}$$
On the other hand, for each $x\in \partial \O_n$  by \eqref{g4} we have 
$$\int_0^1  g'(P_nx(s)) (P_nx)(s)ds \leq \int_0^1 ( \alpha g(P_nx(s)) + \beta)ds = \alpha r + \beta ,$$
so that 
$${\mathcal H}_n(x) \leq \frac{\pi^2(\alpha r + \beta)}{8a^2}, \quad x\in \partial \O_n,$$
and also \eqref{curvatura} is satisfied. \end{proof}

Rational functions $g(\xi) := p(\xi)/q(\xi)$, where $q$ is any positive polynomial of degree $n\in \N$, $p$ is any polynomial of degree $n+1$, and $g'(\xi)\neq 0$ for each $\xi\in \R$, satisfy the assumptions of Proposition \ref{strano}. 
 
 \appendix

 \section{The law of the Ornstein--Uhlenbeck process in $C([0,T];X)$}

We study the law of the solution $X(\cdot, x)$ to \eqref{e1.6} in $C([0,T];X)$, for a fixed $T>0$. Although the next results should be more or less known, we were not able to find any simple proper reference. 

Let $(W_k)$ be
sequence of  mutually independent real Brownian processes defined on
a filtered  probability space $(\Omega,\mathcal F, (\mathcal F_t)
_{t\ge 0},\P)$.  Define
\begin{equation}
\label{e1}
Z(t):=\sqrt Q\;W(t)=\sum_{k=1}^\infty \lambda_k^{1/2}
W_k(t)e_k. 
\end{equation}

The proof of the following lemma is standard.

\begin{Lemma}
\label{l1}
For each $t\ge 0$ the series   \eqref{e1}
converges in  $L^2(\Omega,\mathcal F,\P;X)$ and  $Z(t)$   is a $X$- valued Gaussian random variable with mean $0$ and
covariance operator $tQ$. Moreover, $\{Z(t):\; t\ge 0\}$ is a Gaussian process, that is for any $0<t_1<\cdots<t_n$, the random variable $(X(t_1),....X(t_n))$, with values in $X^n $, is Gaussian. 
\end{Lemma}

\begin{Proposition}
\label{p2}
 The process  $\{Z(t):\;t\ge 0\}$  is continuous, 
and the law of $Z(\cdot)$ in $C([0,T];X)$ is  Gaussian  
 for any $T>0$. Moreover its support is the whole $C([0,T];X)$.
\end{Proposition}
\begin{proof}  We first show that for each $T>0$ we have
\begin{equation}
\label{e12}
\E\sup_{t\in[0,T]}|Z(t)|^2<\infty
\end{equation}
which implies that $\{Z(:),\;t\ge 0\}$ has a continuous version. Set
$$
Z_N(t)=\sum_{k=1}^N\lambda_k^{1/2}W_k(t)e_k.
$$
Then by a well known martingale inequality (that is: if $R(t),\;t\in [0,T]$ is a $N$-dimensional martingale, then  $\E\sup_{t\in[0,T]}|R(t)|_{\R^N}^2\le 4\E|R(T)|^2$), we have 
$$
\E\sup_{t\in[0,T]}|Z_{N+p}(t)-Z_N(t)|^2\le 4\E|Z_{N+p}(T)-Z_N(T)|^2=4T\sum_{k=N+1}^{N+p}\lambda_k.
$$
Thus $(Z_N(\cdot))$ is convergent in $L^2([0,T];C([0,T];X))$. \bigskip

Now we show that $Z(\cdot)$ is Gaussian.  Let $F $ be an element of the dual space of $C([0,1];X)$. We have to show that $F(Z(\cdot))$ is a real Gaussian random variable. Since the $L^2(\Omega, \P)$--limit of a sequence of real Gaussian random variables  is a Gaussian random variable, it is enough to show that $F(Z_n(\cdot))$ is Gaussian,  where
 $$
 Z_n(t)=\frac{1}{T^n}\sum_{k=0}^n \binom{n}{k} t^k(T-t)^{n-k}Z\bigg(\frac{Tk}{n}\bigg) 
 $$
 are the Bernstein approximations of $Z$, that converge to $Z$ in $C([0,T];X)$. 
For each $\varphi\in C([0,T];X)$ set
 $$
 F_\varphi(x)= F(\varphi(\cdot)x), \quad x\in X.
 $$
 The functional $F_\varphi$ is linear and bounded, so that there exists  $v_\varphi\in X$ such that
 $$
 F_\varphi(x)=\langle x, v_\varphi\rangle,\quad\forall\;x\in X.
 $$
 Therefore, 
 $$
 F(Z_n(\cdot))=\sum_{k=0}^n \langle Z\bigg(\frac{Tk}{n}\bigg), v_{\varphi_k}\rangle,
 $$
 where  
 $$
 \varphi_k(t)= {n\choose k}\frac{ t^k(T-t)^{n-k}}{T^n}.
 $$
 Now the conclusion follows from the fact that  $\{Z(t):\;t\ge 0\}$ is a Gaussian process. \medskip
 
 It remains to show that  the law of $Z(\cdot)$ is full. To this aim  it is enough to show that
 $$
 \P(\|Z-\alpha\|_{\infty}\le r)>0,\quad \forall \alpha\in C([0,T];X).
 $$
For every $n\in \N$ we have  $\{ \|Z-\alpha\|_{\infty}\le r\}\subset \{ \|P_n(Z-\alpha)\|_{\infty}\le r/2\} \cap 
  \{ \|(I-P_n)(Z-\alpha)\|_{\infty}\le r/2\}$, so that 
$$
\P\left(\|Z-\alpha\|\ge r\right)\ge \P\left(\|P_n(Z-\alpha)\|\le \frac{r}2\right)\P\left(\|(1-P_n)(Z-\alpha)\|\le \frac{r}2\right)
$$
The first factor is positive for any $n$. So, we have to show that  the second one is positive for suitable $n$. 
We have
$$
\P\left(\|(1-P_n)(Z-\alpha)\|_{\infty}\le \frac{r}2\right)=1-\P\left(\|(1-P_n)(Z-\alpha)\|_{\infty}> \frac{r}2\right)
$$
where
\begin{equation}
\label{stima}
\P\left(\|(1-P_n)(Z-\alpha)\|_{\infty}>\frac{r}2\right)\le \frac{4}{r^2}\E\|(1-P_n)(Z-\alpha)\|_{\infty}^2
\end{equation}
For a.e. $\omega\in \Omega$ the function $Z(\cdot)(\omega)-\alpha $ is continuos, so that its range is compact in $X$ and therefore $\|(I-P_n)Z(t)(\omega)-\alpha\|_X $ goes to zero uniformly as $n\to \infty$. So, $\|(I-P_n)Z(\cdot)(\omega)-\alpha \|_{\infty}^2$ goes to zero a.e. in $\Omega$. Since $\|(I-P_n)(Z(\cdot)(\omega)-\alpha) \|_{\infty}^2
\leq \|Z(\cdot)(\omega) -\alpha) \|_{\infty}^2$, by dominated convergence $\lim_{n\to \infty} \E\|(1-P_n)(Z-\alpha)\|_{\infty}^2 =0$. Therefore, for $n$ large enough the right hand side of \eqref{stima} is less than $1$, and the statement follows. 
\end{proof}

For every $x\in X$ let $X(\cdot, x)$ be the unique solution   of \eqref{e1.6}.  

\begin{Proposition}
\label{App}
For each $x\in X$ and $T>0$, $X(\cdot,x)$ is a Gaussian random variable  with values in $C([0,T];X)$,  and the support of its law is the whole  $C([0,T];X)$. Moreover, the associated Cameron--Martin space is contained in $C([0,T];H)$. 
\end{Proposition}
\begin{proof}
Let 
$$ \varphi_x (t)= e^{-t/2}x, \quad 0\leq t\leq T.$$
Then, $X(\cdot, x) -\varphi_x$ satisfies \eqref{e1.6} with null initial datum, so that 
$$ X(t) - \varphi_x(t) = \int_0^t e^{-(t-s)/2} \sqrt{Q}dW(s):=Y(t), \quad 0\leq t\leq T. $$
The right hand side $Y$  is a centered Gaussian random variable with values in $L^2(0,T;X)$ (e.g., \cite[Thm. 5.2]{DPZ1}).  To prove that it has values in $C([0,T];X)$, we remark that 
for a.e. $\omega\in \Omega$ we have
$$Y(t) =   - \frac12\;\int_0^tY(s)ds+\sqrt Q\;W(t), \quad 0\leq t\leq T, $$
and for such $\omega$ we have
\begin{equation}
\label{e2}
Y(t,x) =   \Gamma (\sqrt Q\;W(\cdot)),
\end{equation}
where   $\Gamma $ is the linear bounded operator in $C([0,T];X)$ that maps any   $v\in C([0,T];X)$ into the solution $z$ of the equation
\begin{equation}
\label{e3}
z(t)= -\frac12\;\int_0^tz(s)ds+v(t), \quad 0\leq t\leq T.
\end{equation}
Since $\Gamma$ is linear and bounded, then $Y(\cdot,x)$ is a Gaussian random variable  with values in $C([0,T];X)$.
Since $   \Gamma  $ is continuously invertible, and the law of $\sqrt Q\;W(\cdot)$ is full, then  the law of   $Y(\cdot,x)$   is full, and the law of $X(\cdot, x)$ is full as well. 

To determine the Cameron--Martin space, it is convenient to replace $C([0,T];X)$ by $L^2(0,T;X)$. The advantage is that $L^2(0,T;X)$ is a separable Hilbert space, and for any Gaussian measure  in a separable Hilbert space with mean $m$ and covariance ${\mathcal Q}$,  the Cameron--Martin space is precisely the range of ${\mathcal Q}^{1/2}$. On the other hand, by \cite[Lemma 3.2.2]{Bo} if a separable Banach space $\mathcal X$ is continuously embedded in another Banach space $\mathcal Y$ and $\gamma$ is a Gaussian measure on $\mathcal X$, the Cameron--Martin space is independent whether $\gamma$ is considered on $\mathcal X$ or on $\mathcal Y$. 

In our case,  the mean of $X$ is $\varphi_x$ and the covariance is the operator 
$$
\mathcal Q h(s)=\int_0^Tg(t,s)Qh(s)ds
$$
with 
$$
g(t,s)= \int_0^{\min\{t,s\}}e^{-\frac12(t-r)}e^{-\frac12(s-r)}dr= e^{-\frac12(t+s)} (e^{\min\{t,s\}}-1).
$$
See \cite[Thm. 5.2]{DPZ1}.  
An easy computation shows that the range of $\mathcal Q :L^2(0,T;X)\mapsto  L^2(0,T;X)$ consists of the functions $Qy$, with  
$y\in W^{2,2}(0,T;X)$ and $y(0)=0$, $2y'(T) + y(T) =0$. If $\eta \in $ Range $\mathcal Q$, $\eta = Qy$, then 
$$(\mathcal Q^{-1}\eta )(t) = -y''(t) -\frac{1}{4}y(t) = - Q^{-1}(\eta''(t)+ \frac{1}{4}\eta(t)), \quad 0<t<T.$$
The Cameron--Martin space is the domain of ${\mathcal Q}^{-1/2}$, which is the closure of the range of 
$\mathcal Q$ in the norm $\eta \mapsto \langle {\mathcal Q}^{-1 }\eta, \eta\rangle_{L^2(0,T;X)}$. We have
$$\langle {\mathcal Q}^{-1 }\eta, \eta\rangle_{L^2(0,T;X)} = \int_0^T \langle -Q^{-1}\eta''(t) -\frac{1}{4} Q^{-1}\eta(t), \eta(t)\rangle_X dt$$
$$= \int_0^T ( \|Q^{-1/2} \eta '(t)\|^2_X + \frac{1}{4} \|Q^{-1/2} \eta (t)\|^2_X)dt + \frac{1}{2} \|Q^{-1/2} \eta (T)\|^2_X. $$
Such norm is equivalent to the norm of $W^{1,2}(0,T;H)$. Since $W^{1,2}(0,T;H)\subset C([0,T];H)$, the last statement follows. 
\end{proof}


\begin{thebibliography}{99}

\bibitem{AM} H. Airault, P. Malliavin, \textit{Int\'egration g\'eom\'etrique sur l'espace de Wiener}, Bull. Sci. Math. (2) {\bf 112} (1988), no. 1, 3--52.
 
 
 

 \bibitem{BDPT1} V. Barbu, G. Da Prato,  L. Tubaro, \textit{Kolmogorov equation associated to the stochastic reflection problem
on a smooth convex set of a Hilbert space}, Ann. Probab. {\bf 37} (2009), 1427--1458. 

\bibitem{BDPT2} V. Barbu, G. Da Prato,  L. Tubaro, \textit{Kolmogorov equation associated to the stochastic reflection problem
on a smooth convex set of a Hilbert space II}, Ann. Inst. H. Poincar\`e (to appear).

\bibitem{Bill} P. Billingsley, \textit{Convergence of Probability Measures}, Wiley, New York, 1968.

\bibitem{Bo}  V.I. Bogachev, \textit{Gaussian Measures}, Mathematical Surveys and  Monographs, vol. 62,
American Mathematical Society, Providence, RI, 1998.

\bibitem{Boga2} V.I. Bogachev, \textit{ Differentiable measures and the Malliavin calculus},  Mathematical Surveys and  Monographs, vol. 164,  American Mathematical Society, Providence, RI, 2010.
 
\bibitem{Cerrai}  S. Cerrai,   \textit{Second order PDE's in finite
and infinite dimensions. A probabilistic approach}, Lecture Notes in Mathematics,     
{\bf 1762}, Springer-Verlag, 2001.

\bibitem{CG} A. Chojnowska-Michalik,
B. Goldys,  \textit{Symmetric Ornstein-Uhlenbeck operators and their generators}, Probab. Theory Relat. Fields 124, 459--486 (2002). 




\bibitem{DPGZ} G. Da Prato, B. Goldys,  J. Zabczyk,  
\textit{Ornstein-Uhlenbeck semigroups in open sets of Hilbert spaces},
 C. R. Acad. Sci. Paris,  {\bf 325}, 433--438, 1997.

\bibitem{DPL} G. Da Prato, A. Lunardi,  \textit{On the Dirichlet semigroup for Ornstein--Uhlenbeck 
operators in subsets of Hilbert spaces}, J. Funct. Anal. 259 (2010), 2642--2672.

\bibitem{DPZ1} G. Da Prato,  J. Zabczyk,   \textit{Stochastic equations
  in infinite dimensions}, Cambridge University Press, Cambridge, 1992.

\bibitem{DPZ3} G. Da Prato,  J. Zabczyk,  \textit{Second Order Partial Differential Equations in Hilbert Spaces},   
London  Mathematical Society, Lecture Notes, {\bf 293}, Cambridge University Press, Cambridge, 2002. 





\bibitem{LMP} A. Lunardi, G. Metafune,  D. Pallara,  \textit{Dirichlet boundary conditions for elliptic operators with unbounded drift}, Proc. Amer. Math. Soc. {\bf 133} (2005), 2625--2635. 


\bibitem{MVN} J. Maas, J. Van Neerven,  \textit{Boundedness of Riesz transforms for elliptic operators in abstract Wiener spaces},   J.  Funct. Anal.  {\bf 257} (2009),  2410--2475. 


\bibitem{M} P. Malliavin,   \textit{Stochastic analysis},   Springer-Verlag, Berlin, 1997.

\bibitem{Me}  P.-A. Meyer, \textit{Note sur les processus d'Ornstein--Uhlenbeck}, in: ``S\'eminaire de Probabilit\'es, XVI",  Lecture Notes in Math., vol. 920, Springer, Berlin, 1982, pp. 95--133. 

\bibitem{MPPS}  G. Metafune, J. Pr\"uss, R. Schnaubelt, A. Rhandi, \textit{$L^p$-regularity for elliptic operators with unbounded coefficients}, Adv. Differential Equations {\bf 10} (2005), 1131--1164. 



\bibitem{S} I. Shigekawa,  \textit{Sobolev spaces over the Wiener space based on an Ornstein-Uhlenbeck operator},
J. Math. Kyoto Univ. {\bf 32} (1992), 731--748. 
   
\bibitem{Tal} A. Talarczyk, \textit{Dirichlet problem for parabolic equations on Hilbert spaces},  Studia Math. {\bf 141} (2000), 109--142.
 

\end{thebibliography}
\end{document}